\documentclass[a4paper,english,11pt]{article}

\usepackage{geometry}
\usepackage{pdfpages}
\usepackage[T1]{fontenc}
\usepackage[utf8]{inputenc}
\usepackage{babel}
\usepackage{amsmath}
\usepackage{amssymb}
\usepackage{amsthm}
\usepackage{textcomp}
\usepackage{enumitem}
\usepackage{dsfont}
\usepackage{mathrsfs}
\usepackage{cite}
\usepackage{color}
\usepackage{breqn}
\usepackage[colorlinks=true,urlcolor=blue,linkcolor=blue]{hyperref}

\newcommand{\R}{\mathbb{R}}

\newcommand{\Z}{\mathbb{Z}}
\newcommand{\N}{\mathbb{N}}
\newcommand{\C}{\mathbb{C}}

\newcommand{\T}{\mathbb{T}}

\newcommand{\holo}{\textnormal{Hol}}

\newcommand{\heis}{\mathbb{H}^1}

\newcommand{\hlapl}{\Delta_{\heis}}

\newcommand{\id}{\textnormal{Id}}

\newcommand{\classeC}{\mathcal{C}}

\renewcommand{\Im}{\textnormal{Im}}
\renewcommand{\Re}{\textnormal{Re}}
\newcommand{\un}{\mathds{1}}
\newcommand{\hardy}{\mathcal{H}^2(\mathbb{C}_+)}
\newcommand{\scw}{\mathscr{S}}
\newcommand{\half}{\frac{1}{2}}

\newcommand{\Qbeta}{\mathcal{Q}_{\beta}}

\newcommand{\M}{\mathcal{M}}

\renewcommand\d{\,{\mathrm d}}
\newcommand\e{\,{\mathrm e}}
\newcommand{\gdO}{{\mathcal O}}

\newtheorem{mydef}{Definition}[section]
\newtheorem{thm}[mydef]{Theorem}
\newtheorem{lem}[mydef]{Lemma}
\newtheorem{prop}[mydef]{Proposition}

\title{On the orbital stability of a family of traveling waves for the cubic Schrödinger equation on the Heisenberg group.}
\author{Louise Gassot}
\date{}
\newcommand{\Addresses}{{
  \bigskip
  \footnotesize
\noindent
  \textsc{Département de mathématiques et applications, École normale supérieure, CNRS, PSL University, 75005 Paris, France\\Laboratoire de Mathématiques d’Orsay, Univ. Paris-Sud, CNRS, Université Paris-Saclay, 91405 Orsay, France}\par\nopagebreak
  \noindent
  \textit{E-mail address :} \texttt{louise.gassot@math.u-psud.fr}
}}
\begin{document}
\maketitle
\abstract{We consider the focusing energy-critical Schrödinger equation on the Heisenberg group in the radial case
\[
i\partial_t u-\Delta_{\mathbb{H}^1} u=|u|^2u,
\quad	\Delta_{\mathbb{H}^1}=\frac{1}{4}(\partial_x^2+\partial_y^2)+(x^2+y^2)\partial_s^2,
\quad	(t,x,y,s)\in \mathbb{R}\times\mathbb{H}^1,
\]
which is a model for non-dispersive evolution equations. For this equation, existence of smooth global solutions and uniqueness of weak solutions in the energy space are open problems. We are interested in a family of ground state traveling waves parametrized by their speed $\beta\in(-1,1)$. We show that the traveling waves of speed close to $1$ present some orbital stability in the following sense. If the initial data is radial and close enough to one traveling wave, then there exists a global weak solution which stays close to the orbit of this traveling wave for all times. A similar result is proven for the limiting system associated to this equation.
}

\tableofcontents

\newpage
\section{Introduction}

\subsection{Motivation}\label{subsection:motivation}
We are interested in the Schrödinger equation on the Heisenberg group
\begin{equation}\label{eq:H}
\begin{cases}
i\partial_t u-\hlapl u=|u|^2u
\\
u(t=0)=u_0
\end{cases},\quad (t,x,y,s)\in\R\times\heis.
\end{equation}
The operator $\hlapl$ denotes the sub-Laplacian on the Heisenberg group. When the solution is radial, in the sense that it only depends on $t$, $|x+iy|$ and $s$, the sub-Laplacian writes
\[
\hlapl =\frac{1}{4}(\partial_x^2+\partial_y^2)+(x^2+y^2)\partial_s^2.
\]

The Heisenberg group is a typical case of sub-Riemannian geometry where dispersive properties of the Schrödinger equation disappear (see Bahouri, Gérard and Xu \cite{BahouriGerardXu2000}).
 To take it further, Del Hierro \cite{DelHierro2005} proved sharp decay estimates for the Schrödinger equation on H-type groups, depending on the dimension of the center of the group. More generally, Bahouri, Fermanian and Gallagher \cite{BahouriFermanianGallagher2016} proved optimal dispersive estimates on stratified Lie groups of step $2$ under some property of the canonical skew-symmetric form. In contrast, they also give a class of groups without this property displaying total lack of dispersion, which includes the Heisenberg group.

Dispersion impacts the way one can address the Cauchy problem for the Schrödinger equation. Indeed (see Burq, Gérard and Tzvetkov \cite{Burq2005}, remark 2.12), the existence of a smooth local in time flow map on the Sobolev space $H^s(M)$ for the Schrödinger equation on a Riemaniann manifold $M$ with Laplace-Beltrami operator $\Delta$
\[
\begin{cases}
i\partial_t u-\Delta u=|u|^2u
\\
u(t=0)=u_0
\end{cases}
\]
implies the following Strichartz estimate
\[
\|\e^{it\Delta}f\|_{L^4([0,1]\times M)}\leq C\|f\|_{H^{\frac{s}{2}}(M)}.
\]
The argument also applies for the Heisenberg group with the homogeneous Sobolev spaces $\dot{H}^s(\heis)$, for which the inequality holds if and only if $s\geq 2$ \cite{GerardGrellier2008}. In particular, uniqueness of weak solutions in the energy space $\dot{H}^1(\heis)$ is an open problem, whereas without a conservation law controlling the $\dot{H}^2$-norm, there is no existence existence result of smooth global solutions. Constructing weak global solutions to the Schrödinger equation on the Heisenberg group is still possible through a compactness argument, nevertheless, this method does not ensure that these solutions are relevant. Indeed, the energy of the solution is only bounded above by the initial energy. Therefore, the cancellation of the energy of the solution at some time may not imply that the solution is identically zero, and we do not exclude the possibility of non-uniqueness of weak solutions as in the 2D incompressible Euler equation \cite{DeLellisSzekelyhidi2009}.

The aim of this paper is to construct some weak global solutions with a prescribed behaviour. More precisely, given an initial data close to some ground state traveling wave solution for the Schrödinger equation on the Heisenberg group, we want to construct a weak global solution which stays close to the orbit of the traveling wave for all times. Combined with an uniqueness result, this would lead to the orbital stability of this ground state traveling wave.

\subsection{Main results}

We consider a family of traveling waves with speed $\beta\in(-1,1)$ under the form
\[
u_\beta(t,x,y,s)=\sqrt{1-\beta}Q_\beta(x,y,s+\beta t).
\]
The profile $Q_\beta$ satisfies the following stationary hypoelliptic equation
\begin{equation*}
-\frac{\hlapl +\beta D_s}{1-\beta}Q_\beta=|Q_\beta|^2Q_\beta.
\end{equation*}
Because of the scaling invariance, it would have been equivalent in the rest of the study to define $Q_\beta$ as
\[
u_\beta(t,x,y,s)=Q_\beta\left(\frac{x}{\sqrt{1-\beta}},\frac{y}{\sqrt{1-\beta}},\frac{s+\beta t}{1-\beta}\right).
\]
From \cite{schroheis}, we know that as $\beta$ tends to $1$, the ground state solutions of speed $\beta$ converge up to symmetries in $\dot{H}^1(\heis)$ to some profile $Q$. Moreover, $Q$ is solution to a limiting equation
\begin{equation}\label{eq:Hplus}
D_sQ=\Pi_0^+(|Q|^2Q),
\end{equation}
for which the ground state solution is unique up to symmetries, equal to
\[
Q(x,y,s)=\frac{i\sqrt{2}}{s+i(x^2+y^2)+i}.
\]
The operator $\Pi_0^+$ is an orthogonal projector onto a relevant space for our analysis denoted by $V_0^+$. For more details, see the Notation part \ref{notation}.

We first prove the conditional orbital stability of the ground state $Q$ in the limiting equation, and then focus on the conditional orbital stability of the ground states $Q_\beta$ in the Schrödinger equation on the Heisenberg group as $\beta$ is close to $1$.

\begin{mydef}
Fix $u\in\dot{H}^1(\heis)$ and $X=(s_0,\theta,\alpha)\in\R\times\T\times\R_+^*$. We denote by $T_Xu$ the element of $\dot{H}^1(\heis)$ satisfying
\[
T_Xu(x,y,s):=\e^{i\theta}\alpha u(\alpha x,\alpha y,\alpha^2(s-s_0)), \quad (x,y,s)\in\heis.
\]
We also define
\[
d(u,\M)=\inf_{X\in \R\times\T\times\R_+^*}\|T_Xu-Q\|_{\dot{H}^1(\heis)},
\]
as the distance of $u$ to the orbit $\M$ of $Q$
\[
\M=\{T_XQ\mid X\in \R\times\T\times\R_+^*\},
\]
and
\[
d(u,\sqrt{1-\beta}\Qbeta)=\inf_{X\in \R\times\T\times\R_+^*}\|T_Xu-\sqrt{1-\beta}Q_\beta\|_{\dot{H}^1(\heis)},
\]
as the distance of $\frac{u}{\sqrt{1-\beta}}$ to the orbit $\Qbeta$ of $Q_\beta$
\[
\Qbeta=\{T_XQ_\beta \mid X\in\R\times\T\times\R_+^*\}.
\]
\end{mydef}

Our first result is that the profile $Q$ is conditionally orbitally stable for the evolution problem linked to the limiting equation
\begin{equation}\label{eq:bergman_evolution}
\begin{cases}
i\partial_t u=\Pi_0^+(|u|^2u)
\\
u(t=0)=u_0
\end{cases}.
\end{equation}
\begin{thm}[Conditional orbital stability of $Q$]\label{thm:stability_Q}
For some $c_0>0$ and $r_0>0$ the following holds. Let $r\leq r_0$ and $u_0\in\dot{H}^1(\heis)\cap V_0^+$ such that
\begin{equation}\label{eq:initial_cond_Q}
\|u_0-Q\|_{\dot{H}^1(\heis)}< r^2.
\end{equation}
Then there exists a weak solution $u\in\classeC(\R,\dot{H}^1(\heis))$ (with the weak topology) to equation \eqref{eq:bergman_evolution} such that for all $t\in\R$,
\[
d(u(t),\M)\leq c_0r.
\]
\end{thm}

Using the links between the limiting equation and the Schrödinger equation, we deduce our second result : the conditional orbital stability of the profiles $Q_\beta$ for the Schrödinger equation when $\beta$ is close to $1$ in the radial case.
\begin{thm}[Conditional orbital stability of $Q_\beta$]\label{thm:stability_Qbeta}
For some $c_0>0$ and $r_0>0$ the following holds. Let $r\in(0,r_0)$. Then there exists $\beta_*\in(0,1)$ such that if $\beta\in(\beta_*,1)$, and if $u_0\in\dot{H}^1(\heis)$ is radial and satisfies
\begin{itemize}
\item if $u_0\in\dot{H}^1(\heis)\cap V_0^+$ :
\[
\|u_0-\sqrt{1-\beta}Q_\beta\|_{\dot{H}^1(\heis)}< \sqrt{1-\beta}r^2
\]
\item in the general case :
\begin{equation}\label{eq:intial_cond_Qbeta}
\|u_0-\sqrt{1-\beta}Q_\beta\|_{\dot{H}^1(\heis)}< (1-\beta)r,
\end{equation}
\end{itemize}
then there exists a weak radial solution $u\in\classeC(\R,\dot{H}^1(\heis))$ (with the weak topology) to the Schrödinger equation on the Heisenberg group \eqref{eq:H}
\begin{equation*}
\begin{cases}
i\partial_t u-\hlapl u=|u|^2u
\\
u(t=0)=u_0
\end{cases}
\end{equation*}
such that for all $t\in\R$, $\frac{u(t)}{\sqrt{1-\beta}}$ is close to the orbit of $Q_\beta$~:
\[
d\left(u(t),\sqrt{1-\beta}\Qbeta\right)\leq c_0\sqrt{1-\beta}r.
\]
\end{thm}

Note that, unlike the weak solutions discussed in the first part \ref{subsection:motivation}, the energy of the weak solutions from Theorem \ref{thm:stability_Q} (resp.\@ Theorem \ref{thm:stability_Qbeta}) is controlled, indeed, this energy is very close to the one of the ground state $Q$ (resp.\@ $\sqrt{1-\beta}Q_\beta$). Furthermore, these two theorems would imply the orbital stability of $Q$ and $Q_\beta$ in the radial case in both situations if we had a uniqueness result for the solutions.

The assumption required on a general initial condition for the Schrödinger equation \eqref{eq:intial_cond_Qbeta} is owing to the estimates on the component of the initial condition on the orthogonal of $V_0^+$. In the assumptions of Theorem \ref{thm:stability_Q}, this component is naturally zero, leading to a weaker assumption \eqref{eq:initial_cond_Q} on the initial condition.

The key point in both proofs is the following local stability estimate for $Q$, which comes from the invertibility of the linearized operator around $Q$ for the limiting equation~\eqref{eq:Hplus} on a subspace of $V_0^+$ of finite co-dimension.

\begin{mydef}
For $u\in\dot{H}^1(\heis)\cap V_0^+$, we define
\begin{align*}
\delta(u)
	&:=\Big|\|u\|_{\dot{H}^1(\heis)}^2-\|Q\|_{\dot{H}^1(\heis)}^2\Big|+\Big|\|u\|_{L^4(\heis)}^4-\|Q\|_{L^4(\heis)}^4\Big|\\
	&=|\mathcal{P}(u)-\mathcal{P}(Q)|+|\mathcal{E}(u)-\mathcal{E}(Q)|.
\end{align*}
\end{mydef}

\begin{prop}\cite{schroheis}\label{prop:estimates_stability}
There exist $\delta_0>0$ and $C>0$ such that for all $u\in\dot{H}^1(\heis)\cap V_0^+$, if $\delta(u)\leq \delta_0$, then
\[
d(u,\M)^2
	\leq C\delta(u).
\]
\end{prop}

In order to prove Theorem \ref{thm:stability_Q}, we construct the weak solution for the limiting initial value problem~\eqref{eq:bergman_evolution} as a limit of smooth functions. The approximating functions solve slightly modified equations where we have cut frequencies, so that the Cauchy problem is globally well-posed. We show that we can control their distance to the orbit of the ground state $Q$ using Proposition~\ref{prop:estimates_stability}. Finally, we build modulation parameters which stay bounded on finite time intervals for the approximate solutions, and through a compactness argument, we control the distance of the weak solution to the orbit of $Q$ when passing to the limit.

For Theorem \ref{thm:stability_Qbeta}, the idea for the construction is the same, however we only have at our disposal the information on the limiting equation from Proposition~\ref{prop:estimates_stability}. Therefore, we need to take advantage of the fact that $Q_\beta$ is close to $Q$ when $\beta$ is close to~$1$. In this spirit, in order to tackle Theorem \ref{thm:stability_Qbeta} for the speed $\beta$, we first introduce Cauchy problems for the Schrödinger equation \eqref{eq:H} with a parameter $\gamma$ increasing from $\beta$ to $1$. We display some continuity between the Cauchy problems, therefore it is possible to show their convergence to a Cauchy problem for the limiting equation as $\gamma$ tends to $1$. In the proof, we combine this strategy with the above method : we approximate by smooth functions the weak solutions to the Cauchy problems with parameter $\gamma$ by cutting frequencies. Finally, we are able to get back to the problem with speed $\beta$ by continuity and conclude in the same way as the proof of Theorem~\ref{thm:stability_Q}, by constructing bounded modulation parameters for the approximate solutions.

\subsection{Comparison with other equations}
Concerning the focusing energy-critical Schrödinger equation on the Euclidean plane $\R^N$
\[
i\partial_t u-\Delta u=|u|^{p_c-1}u,
\]
where $N\geq 3$ and $p_c=\frac{N+2}{N-2}$, there exists an explicit stationary solution
\[
W(x)=\frac{1}{(1+\frac{|x|^2}{N(N-2)})^{\frac{N-2}{2}}}.
\]
The orbit $\{x\mapsto CW(\frac{x+x_0}{\lambda}) \mid (C,x_0,\lambda)\in\R\times\R\times\R_+^*\}$ of $W$ is the set of minimizers for the Sobolev embedding $\dot{H}^1(\R^N)\hookrightarrow L^{2^*}(\R^N)$ (see the work of Talenti\cite{Talenti1976} and Aubin \cite{Aubin1976}). The energy $E(W)=\half \|W\|_{\dot{H}^1(\R^N)}-\frac{1}{p_c+1}\|W\|_{L^{p_c+1}(\R^N)}$ and the $\dot{H}^1$ norm $\|W\|_{\dot{H}^1(\R^N)}$ play an important role in the dynamical behaviour of the solutions. Kenig and Merle \cite{KenigMerle2006} proved in the radial case that if $N\in\{3,4,5\}$ and the initial condition $u_0\in\dot{H}^1(\R^N)$ satisfies $E(u_0)<E(W)$ and $\|u_0\|_{\dot{H}^1(\R^N)}<\|W\|_{\dot{H}^1(\R^N)}$, then the solution is global and scatters in $\dot{H}^1(\R^N)$, whereas if $E(u_0)<E(W)$ and $\|u_0\|_{\dot{H}^1(\R^N)}>\|W\|_{\dot{H}^1(\R^N)}$, then the solution must blow up in finite time.

The situation is different for the Schrödinger equation on the Heisenberg group. Indeed, from the equation satisfied $Q_\beta$, one can see that the traveling waves
\[
u_\beta(t,x,y,s)=\sqrt{1-\beta}Q_\beta(x,y,s+\beta t)
\]
have a vanishing energy as $\beta$ tends to $1$ :
\[
E(u_\beta(t))=\half\|u_\beta(t)\|_{\dot{H}^1(\heis)}^2-\frac{1}{4}\|u_\beta(t)\|_{L^4(\heis)}^4
	\sim (1-\beta)\frac{\pi^2}{2}
	\to 0,
\]
therefore there exists solutions that do not scatter with arbitrary small energy.

A better parallel would be the mass-critical focusing half-wave equation on the real line
\begin{equation}\label{eq:halfwave}
i\partial_t u+|D|u=|u|^2u, \quad (t,x)\in\R\times \R,
\end{equation}
where $D=-i\partial_x$, $\widehat{|D|f}(\xi)=|\xi|\widehat{f}(\xi)$. The half-wave equation in one dimension also presents some lack of dispersion, and admits traveling waves with speed $\beta\in(-1,1)$ (see Krieger, Lenzmann and Raphaël \cite{KriegerLenzmannRaphael2013})
\[
u(t,x)=Q_\beta\Big(\frac{x+\beta t}{1-\beta}\Big)\e^{-it},
\]
where the profile $Q_\beta$ is a solution to
\[
\frac{|D|-\beta D}{1-\beta}Q_\beta+Q_\beta=|Q_\beta|^2Q_\beta.
\]
The profiles $Q_\beta$ in the half-wave equation converge \cite{GerardLenzmannPocovnicuRaphael2018} as $\beta$ tends to $1$ in $H^\half(\R)$ to a ground state solution $Q$ to some limiting equation
\[
DQ+Q=\Pi(|Q|^2Q), \quad D=-i\partial_x.
\]
From $Q$, we recover a traveling wave solution to the cubic Szeg\H{o} equation
\begin{equation}\label{eq:szego}
i\partial_t u=\Pi(|u|^2u)
\end{equation}
by setting
\(
u(t,x)=Q(x-t)\e^{-it}.
\)
Moreover, the linearized operator around $Q$ is coercive \cite{Pocovnicu2012}, and in particular, the Szeg\H{o} profile is orbitally stable in the relevant space for $Q$
\[
H^\half_+(\R)=\{u\in H^\half(\R)\mid \mathrm{Supp}(\widehat{u})\subset \R_+\}.
\]
\begin{thm}[Orbital stability of $Q$ for the Szeg\H{o} equation]
There exist $\varepsilon_0>0$ and $C>0$ such that for all solution $u$ of the Szeg\H{o} equation \eqref{eq:szego} with initial condition $u_0\in H^\half_+(\R)$, if
\[
\|u_0-Q\|_{H^{\half}(\R)}\leq\varepsilon_0,
\]
then
\[
\sup_{t\in\R}\inf_{(\gamma,y)\in\T\times\R}\|e^{-i\gamma}u(t,\cdot-y)-Q\|_{H^{\half}(\R)}^2
	\leq C\|u_0-Q\|_{H^{\half}(\R)}.
\]
\end{thm}
Gérard, Lenzmann, Pocovnicu and Raphaël \cite{GerardLenzmannPocovnicuRaphael2018} deduced the invertibility of the linearized operator for the half-wave equation around the profiles $Q_\beta$ when $\beta$ is close enough to $1$, and their estimates imply the orbital stability of theses profiles.
\begin{thm}[Orbital stability of $Q_\beta$ for the half-wave equation]
There exists $\beta_*\in(0,1)$ such that the following holds. Let $\beta\in(\beta_*,1)$. Then there exist $\varepsilon_0(\beta)>0$ and $C(\beta)>0$ such that for all solution $u$ of the half-wave equation \eqref{eq:halfwave} with initial condition $u_0\in H^\half(\R)$, if
\[
\|u_0-Q_\beta(\frac{\cdot}{1-\beta})\|_{H^{\half}(\R)}\leq\varepsilon_0(\beta),
\]
then
\[
\sup_{t\in\R}\inf_{(\gamma,y)\in\T\times\R}\|e^{-i\gamma}u(t,\cdot-y)-Q_\beta(\frac{\cdot}{1-\beta})\|_{H^{\half}(\R)}^2
	\leq C\|u_0-Q_\beta(\frac{\cdot}{1-\beta})\|_{H^{\half}(\R)}.
\]
\end{thm}
In higher dimensions $d\geq 2$, traveling waves for the half-wave equation on $\R^d$
\begin{equation*}
i\partial_t u+\sqrt{-\Delta}u=|u|^{p-1}u, \quad (t,x)\in\R\times \R^d,
\end{equation*}
are also orbitally stable in the radial case for mass-subcritical non-linearities $1<p<1+\frac{2}{n}$, but orbitally unstable in the mass-supercritical regime $1+\frac{2}{n}<p<1+\frac{2}{n-1}$ \cite{BellazziniGeorgievVisciglia2018}. Moreover, in the energy-critical and subcritical case, Bellazzini, Georgiev, Lenzmann and Visciglia \cite{BellazziniGeorgievLenzmannVisciglia2019} proved that there can be no small data scattering in the energy space because of the existence of traveling waves with arbitrary small energy.

As we will see in this paper, one cannot directly adapt the proofs for the half-wave equation because we lack information on the Cauchy problem. A second complication arising in comparison to the half-wave equation is the fact that only two conservation laws are available (energy and momentum), because the masses of the ground states may be infinite (this fact is easy to check for $Q$ for instance). The method both for the Schrödinger equation on the Heisenberg group and for its limiting system is the construction of some weak solutions as a limit of smooth functions, and show that we can pass to the limit on their stability properties.

The paper is organized as follows. We first prove the orbital stability of $Q$ for the limiting equation in Section \ref{section:stability_Q}. Then, we assess how close the solutions are to the limiting equation as $\beta$ tends to $1$, in order to study the orbital stability of $Q_\beta$ for the Schrödinger equation in Section \ref{section:stability_Qbeta}.

\paragraph{Acknowledgements} The author is grateful to her PhD advisor P. Gérard for his generous advice and encouragement.
\section{Notation}\label{notation}

\subsection{The Heisenberg group}

Let us now recall some facts about the Heisenberg group. We use coordinates and identify the Heisenberg group $\heis$ with $\R^3$. The group multiplication is given by
\[
(x,y,s)\cdot(x',y',s')=(x+x',y+y',s+s'+2(x'y-xy')).
\]
The Lie algebra of left-invariant vector fields on $\heis$ is spanned by the vector fields $X=\partial_x+2y\partial_s$, $Y=\partial_y-2x\partial_s$ and $T=\partial_s=\frac{1}{4}[Y,X]$. The sub-Laplacian is defined as
\begin{align*}
\mathcal{L}_0 
	:=\frac{1}{4}(X^2+Y^2)
	=\frac{1}{4}(\partial_x^2+\partial_y^2)+(x^2+y^2)\partial_s^2+(y\partial_x -x\partial_y)\partial_s.
\end{align*}
When $u$ is a radial function, the sub-Laplacian coincides with the operator
\[
\hlapl :=\frac{1}{4}(\partial_x^2+\partial_y^2)+(x^2+y^2)\partial_s^2.
\]

The space $\heis$ is endowed with a smooth left invariant measure, the Haar measure, which in the coordinate system $(x,y,s)$ is the Lebesgue measure $\d\lambda_3(x,y,s)$. Sobolev spaces of positive order can then be constructed on $\heis$ from powers of the operator $-\hlapl$, for example, $\dot{H}^1(\heis)$ is the completion of the Schwarz space $\scw(\heis)$ for the norm
\[
\|u\|_{\dot{H}^1(\heis)}:=\|(-\hlapl)^{\frac{1}{2}}u\|_{L^2(\heis)}.
\]

\subsection{Decomposition along the Hermite functions}

In order to study radial functions valued on the Heisenberg group $\heis$, it is convenient to use their decomposition along Hermite-type functions (see for example \cite{stein_harmonic}, Chapters 12 and 13). The Hermite functions
\[
h_m(x)=\frac{1}{\pi^{\frac{1}{4}}2^{\frac{m}{2}}(m!)^{\half}}(-1)^m\e^{\frac{x^2}{2}}\partial_x^m(\e^{-x^2}), \quad x\in\R, m\in\N,
\]
form an orthonormal basis of $L^2(\R)$. In $L^2(\R^2)$, the family of products of two Hermite functions $(h_m(x)h_p(y))_{m,p\in\N}$ diagonalizes the two-dimensional harmonic oscillator~: for all $m,p\in\N$,
\[
(-\Delta_{x,y}+x^2+y^2)h_m(x)h_p(y)=2(m+p+1)h_m(x)h_p(y).
\]

Given $u\in \scw(\heis)$, we will denote by $\widehat{u}$ its usual Fourier transform under the $s$ variable, with corresponding variable $\sigma$
\[
\widehat{u}(x,y,\sigma)=\frac{1}{\sqrt{2\pi}}\int_{\R}\e^{-is\sigma}u(x,y,s)\d s.
\]
For $m,p\in\N$, set $\widehat{h_{m,p}}(x,y,\sigma):=h_m(\sqrt{2|\sigma|}x)h_p(\sqrt{2|\sigma|}y)$. Then the family $(h_{m,p})_{m,p\in\N}$ diagonalizes the sub-Laplacian in the following sense :
\[
\widehat{\hlapl h_{m,p}}=-(m+p+1)|\sigma|\widehat{h_{m,p}}.
\]
Let $k\in\{-1,0,1\}$, and denote by $\dot{H}^k(\heis)\cap V_n^{\pm}$ the subspace of $\dot{H}^k(\heis)$ spanned by $\{h_{m,p}\mid m,p\in\N, m+p=n\}$. Some $u_n^{\pm}\in \dot{H}^k(\heis)$ belongs to $\dot{H}^k(\heis)\cap V_n^{\pm}$ if there exists a family $(f_{m,p}^\pm)_{m+p=n}$ such that
\[
\widehat{u_n^\pm}(x,y,\sigma)=\sum_{\substack{m,p\in\N;\\ m+p=n}}f_{m,p}^{\pm}(\sigma)\widehat{h_{m,p}}(x,y,\sigma)\un_{\sigma\gtrless 0}.
\]
For $u_n^\pm\in\dot{H}^k(\heis)\cap V_n^{\pm}$, the $\dot{H}^k$-norm of $u_n^\pm$ writes 
\begin{align*}
\|u_n^{\pm}\|_{\dot{H}^k(\heis)}^2
	&=\int_{\R_{\pm}}((n+1)|\sigma|)^k\int_{\R^2}|\widehat{u_n^\pm}(x,y,\sigma)|^2\d x\d y\d\sigma\\
	&=\sum_{\substack{m,p\in\N;\\ m+p=n}}\int_{\R_{\pm}}((n+1)|\sigma|)^k|f_{m,p}^{\pm}(\sigma)|^2\frac{\d\sigma}{2|\sigma|}.
\end{align*}
Any function $u\in\dot{H}^k(\heis)$ admits a decomposition along the orthogonal sum of the subspaces $\dot{H}^k(\heis)\cap V_n^{\pm}$. Let us write $u=\sum_{\pm}\sum_{n\in\N}u_n^{\pm}$ where $u_n^{\pm}\in \dot{H}^k(\heis)\cap V_n^{\pm}$ for all $(n,\pm)$. Then 
\begin{align*}
\|u\|_{\dot{H}^k(\heis)}^2
	&=\sum_{\pm}\sum_{n\in\N}\|u_n^{\pm}\|_{\dot{H}^k(\heis)}^2.
\end{align*}

For $k=0$, we get an orthogonal decomposition of the space $L^2(\heis)$, and denote by $\Pi_n^{\pm}$ the associated orthogonal projectors.

The particular space $\dot{H}^k(\heis)\cap V_0^+$ is spanned by a unique radial function $h_0^+$, satisfying
\[
\widehat{h_0^+}(x,y,\sigma)=\frac{1}{\sqrt{\pi}}\e^{-(x^2+y^2)\sigma}\un_{\sigma\geq 0}.
\]
Set $u\in \dot{H}^k(\heis)\cap V_0^+$, then there exists $f$ such that
\[
\widehat{u}(x,y,s)=f(\sigma)\widehat{h_0^+}(x,y,\sigma),
\]
and
\[
\|u\|_{\dot{H}^k(\heis)}^2=\int_{\R_+}|f(\sigma)|^2\frac{\d\sigma}{2\sigma^{1-k}}.
\]

\section{Conditional orbital stability of the ground state \texorpdfstring{$Q$}{Q} in the limiting equation}\label{section:stability_Q}

In this section, we prove Theorem \ref{thm:stability_Q} on the conditional orbital stability of the ground state $Q$ in the limiting equation \eqref{eq:bergman_evolution}
\begin{equation*}
\begin{cases}
i\partial_t u=\Pi_0^+(|u|^2u)
\\
u(t=0)=u_0\in\dot{H}^1(\heis)\cap V_0^+
\end{cases}, \quad (t,x,y,s)\in\R\times\heis.
\end{equation*}
For convenience, we replace in this part the elements $u\in\dot{H}^k(\heis)\cap V_0^+$, $k\in\{-1,0,1\}$ with the corresponding holomorphic function on the complex upper half-plane $F_u\in\dot{H}^{\frac{k}{2}}(\C_+)\cap\holo(\C_+)$, defined as
\[
F_u(s+i(x^2+y^2)):=u(x,y,s),
\]
which transforms the Cauchy problem for $u$ in a Cauchy problem for $F_u$ written as
\begin{equation}\label{eq:bergman}
\begin{cases}
i\partial_t u=P_0(|u|^2u)
\\
u(t=0)=u_0\in\dot{H}^\half(\C_+)\cap\holo(\C_+)
\end{cases}, \quad (t,z)\in\R\times \C_+,
\end{equation}
$P_0$ being a Bergman projector. We recall such a correspondence in part \ref{subsection:weighted_bergman}. Then we construct some smooth functions approximating a weak solution of equation \eqref{eq:bergman} in part \ref{subsection:approximate_solutions_q}, prove their weak convergence in part \ref{subsection:limiting_weak_convergence}, and deduce from their distance to the orbit of $Q$ an upper bound on the distance of the weak limit to this orbit in part \ref{subsection:limiting_modulation}.

\subsection{Weighted Bergman spaces}\label{subsection:weighted_bergman}

Recall that if $u\in\dot{H}^k(\heis)\cap V_0^+$, then $F_u\in\dot{H}^{\frac{k}{2}}(\C_+)\cap \holo(\C_+)$, and the Fourier transform of $u$ along the $s$ variable corresponds to a function in $L^2(\R_+,\sigma^{k-1}\d\sigma)$ : for some $f\in L^2(\R_+,\sigma^{k-1}\d\sigma)$,
\[
F_u(z)=\frac{1}{\pi\sqrt{2}}\int_0^{+\infty}e^{iz\sigma}f(\sigma)\d\sigma
\]
where
\[
\|u\|_{\dot{H}^k(\heis)}^2=\pi\|F_u\|_{\dot{H}^{\frac{k}{2}}(\C_+)}^2=\pi\|(-i\partial_z)^{\frac{k}{2}}F_u\|_{L^2(\C_+)}^2=\half\int_0^{+\infty}|f(\sigma)|^2\sigma^{k-1}\d\sigma.
\]

For $k<1$, $F_u$ belongs to the weighted Bergman space $A^2_{1-k}$.
\begin{mydef}[Weighted Bergman spaces] Given $k<1$, the weighted Bergman space $A_{1-k}^2$ is the subspace of
$
L^2_{1-k}:=L^2(\C_+,\Im(z)^{-k}\d\lambda(z))
$
composed of holomorphic functions of the complex upper half-plane $\C_+$~:
\[
A_{1-k}^2:=\left\{F\in\holo(\C_+)\mid \|F\|_{L_{1-k}^2}^2:=\int_0^{+\infty}\int_{\R}|F(s+it)|^2\d s\frac{\d t}{t^{k}}<+\infty\right\}.
\]
\end{mydef}

Indeed, recall the Paley-Wiener theorem for Bergman spaces \cite{bergman_projectors}.
\begin{thm}[Paley-Wiener] Let $k<1$. Then for every $f\in L^2(\R_+,\sigma^{k-1}\d\sigma)$, the following integral is absolutely convergent on $\C_+$
\begin{equation}\label{eq:paley1}
F(z)=\frac{1}{\sqrt{2\pi}}\int_0^{+\infty}\e^{iz\sigma}f(\sigma)\d\sigma,
\end{equation}
and defines a function $F\in A^2_{1-k}$ which satisfies
\begin{equation}\label{eq:paley2}
\|F\|^2_{L^2_{1-k}}=\frac{\Gamma(1-k)}{2^{1-k}}\int_{0}^{+\infty}|f(\sigma)|^2\sigma^{k-1}\d\sigma.
\end{equation}
Conversely, for every $F\in A^2_{1-k}$, there exists $f\in L^2(\R_+,\sigma^{k-1}\d\sigma)$ such that \eqref{eq:paley1} and \eqref{eq:paley2} hold.
\end{thm}

For $k=1$, $F_u$ belongs to the Hardy space $\hardy$.
\begin{mydef}
The Hardy space $\hardy$ space of holomorphic functions of the upper half-plane $\C_+$ such that the following norm is finite~:
\[
\|F\|^2_{\hardy}:=\sup_{t>0}\int_{\R}|F(s+it)|^2\d s<+\infty.
\]
\end{mydef}
\begin{thm}[Paley-Wiener] For every $f\in L^2(\R_+)$, the following integral is absolutely convergent on $\C_+$
\begin{equation}\label{eq:paley3}
F(z)=\frac{1}{\sqrt{2\pi}}\int_0^{+\infty}\e^{iz\sigma}f(\sigma)\d\sigma,
\end{equation}
and defines a function $F$ in the Hardy space $\hardy$ which satisfies
\begin{equation}\label{eq:paley4}
\|F\|^2_{\hardy}=\int_{0}^{+\infty}|f(\sigma)|^2\d\sigma.
\end{equation}
Conversely, for every $F\in \hardy$, there exists $f\in L^2(\R_+)$ such that \eqref{eq:paley3} and \eqref{eq:paley4} hold.
\end{thm}

In the following, we will work with the holomorphic representations, the solutions being valued in the Hardy space $\hardy=\dot{H}^{\half}(\C_+)\cap \holo(\C_+)$.

\subsection{Construction of approximate solutions}\label{subsection:approximate_solutions_q} 

Given an initial data $u_0\in\hardy=\dot{H}^{\half}(\C_+)\cap \holo(\C_+)$ close enough to the ground state
\[
Q(z)=\frac{\sqrt{2}}{z+i},
\]
we want to construct a global solution to the Cauchy problem \eqref{eq:bergman}
\begin{equation*}
\begin{cases}
i\partial_t u=P_0(|u|^2u), \quad (t,z)\in\R\times \C_+
\\
u(t=0)=u_0
\end{cases}
\end{equation*}
which stays close to $Q$ (up to symmetries) for all times. The Bergman projection $P_0$ from $L^2(\C_+)$ to $A^2_1$ writes (see eg \cite{bergman_projectors})
\[
P_0(u)(z)=-\frac{1}{\pi}\int_{\R_+}\int_{\R}\frac{1}{(z-s+it)^2}u(s+it)\d s\d t, \quad z\in\C_+.
\]

We approximate $u$ by functions with higher regularity, satisfying equations for which we can use a classical global well-posedness result.

\paragraph{Construction of smoothing projectors $\widetilde{P}_{\varepsilon,M}$ :}
For $\varepsilon,M>0$, we define the projector $\widetilde{P}_{\varepsilon,M}$  as follows. Write $u\in \dot{H}^k(\C_+)\cap\holo(C_+)$, $k\leq \half$ (or $u\in H^k(\C_+)\cap\holo(\C_+)$, $k\geq 0$) as
\[
u(z)=\frac{1}{\sqrt{2\pi}}\int_0^{+\infty}e^{iz\sigma}f(\sigma)\d\sigma,
\]
then
\[
\widetilde{P}_{\varepsilon,M}(u)(z):=\frac{1}{\sqrt{2\pi}}\int_\varepsilon^{M}e^{iz\sigma}f(\sigma)\d\sigma.
\]
This projector cuts the high and low frequencies of $u$, in order to add some regularity on the solutions. It defines a bounded projector from $\dot{H}^k(\C_+)\cap\holo(\C_+)$ to itself for $k\leq\half$, and from $H^k(\C_+)\cap\holo(\C_+)$ to itself for $k\geq 0$.

\paragraph{Construction of a sequence of approximate solutions $(u_n)_n$ :}
We consider $f\in L^2(\R_+)$ such that for all $z\in\C_+$,
\[
u_0(z)=\frac{1}{\sqrt{2\pi}}\int_0^{+\infty}e^{iz\sigma}f(\sigma)\d\sigma,
\]
which satisfies
\[
\|u_0\|_{\dot{H}^\half(\C_+)}^2=\half\|f\|_{L^2(\R_+)}^2.
\]

Let us fix a sequence of positive numbers $(\varepsilon_n)_n$ going to zero, and consider the following initial data belonging to $H^2(\C_+)\cap\holo(\C_+)$
\[
u_0^n(z)
	:=\widetilde{P}_{\varepsilon_n,\frac{1}{\varepsilon_n}}u_0(z)
	=\frac{1}{\sqrt{2\pi}}\int_{\varepsilon_n}^{1/\varepsilon_n}e^{iz\sigma}f(\sigma)\d\sigma.
\]

We denote by $H^2_{\varepsilon}(\C_+)\cap\holo(\C_+)$ the space of functions $u\in H^2(\C_+)\cap\holo(\C_+)$ satisfying $\widetilde{P}_{\varepsilon,\frac{1}{\varepsilon}}(u)=u$. On this space, the $\dot{H}^k$-norms, $k\geq 0$, are equivalent :
\[
\varepsilon^{2k}\|u\|_{L^2(\C_+)}^2
	\leq \|u\|_{\dot{H}^k(\C_+)}^2
	=\half\int_\varepsilon^{1/\varepsilon}\sigma^{2k-1}|f(\sigma)|^2\d\sigma
	\leq \frac{1}{\varepsilon^{2k}}\|u\|_{L^2(\C_+)}^2.
\]

Define the projection $P_0^n$ as
\[
P_0^n
	=\widetilde{P}_{\varepsilon_n,\frac{1}{\varepsilon_n}} \circ P_0.
\]
We consider the following Cauchy problem
\begin{equation}\label{eq:schrodinger_bergman}
\begin{cases}
i\partial_tu_n=P_0^n(|u_n|^2u_n)
\\
u_n(t=0)=u_0^n
\end{cases},
\end{equation}
which is globally well-posed in $ H^2_{\varepsilon_n}(\C_+)\cap\holo(\C_+)$.
\begin{prop}\label{prop:well_posedness_limit}
Let $u_0^n\in H^2_{\varepsilon_n}(\C_+)\cap\holo(\C_+)$. Then there exists a unique solution $u_n\in\classeC^\infty(\R,H^2_{\varepsilon_n}(\C_+)\cap\holo(\C_+))$ of \eqref{eq:schrodinger_bergman} in the distribution sense.
\end{prop}

\begin{proof}
The local existence comes from the Cauchy-Lipschitz theory for PDEs. Indeed, $H^k(\C_+)\cap\holo(\C_+)$ is an algebra as soon as $k>1$, and in this case, $P_0$ extends to a bounded projector from $H^k(\C_+)\cap\holo(\C_+)$ to itself preserving $H^k_{\varepsilon_n}(\C_+)\cap\holo(\C_+)$, therefore $P_0^n$ extends to a bounded projector from $H^k(\C_+)\cap\holo(\C_+)$ onto $H^k_{\varepsilon_n}(\C_+)\cap\holo(\C_+)$. Moreover, the time of existence of the solution only depends on the norm of the initial data in $H^2(\C_+)$. In order to prove that local solutions extend globally in time, it is now enough to show that the $H^2$-norm of the solution stays bounded. Thanks to the equivalence of the $\dot{H}^k$-norms in $H^2_{\varepsilon_n}(\C_+)\cap\holo(\C_+)$, this lies in the fact that equation \eqref{eq:schrodinger_bergman} has conserved momentum
\[
\mathcal{P}(u):=(u,-iu_z)_{L^2(\C_+)}=\|u\|_{\dot{H}^\half(\C_+)}^2.
\]
\end{proof}

The energy
\(
\mathcal{E}(u)=\|u\|_{L^4(\C_+)}^4
\)
is also conserved.

We now show that $u_n(t)$ is close to the orbit $\M$ of the ground state $Q$. Thanks to Proposition \ref{prop:estimates_stability}, it is enough to focus on $\delta(u_n(t))$. But using the conservation laws, we know that for all $t\in\R$,
\[
\delta(u_n(t))=\delta(u_0^n).
\]
Moreover, by construction of $u_0^n$, we know that $\|u_0^n-u_0\|_{\dot{H}^\half(\C_+)}$ tends to $0$ as $n$ tends to $+\infty$, therefore $\delta(u_0^n)$ tends to $\delta(u_0)$. Assume that $\delta(u_0)<\delta_0$, then $\delta(u_0^n)<\delta_0$ after some rank $N$. Thanks to Proposition \ref{prop:estimates_stability}, we deduce that for all $n\geq N$ and $t\in\R$,
\begin{equation}\label{ineq:delta(u_n)}
d(u_n(t),\M)^2\leq C\delta(u_0^n).
\end{equation}

\subsection{Weak convergence}\label{subsection:limiting_weak_convergence}

In this part, we show that $u_n$ has a weak limit $u$, which is a weak solution to equation \eqref{eq:bergman}. In order to do so, we first prove that $t\mapsto\partial_tu_n(t)$ is uniformly bounded in $\dot{H}^{-\half}(\C_+)$, then use Ascoli's theorem.

Because of the conservation of the momentum and the fact that
\(
\mathcal{P}(u_0^n)\leq \mathcal{P}(u_0)
\)
for all $n\in\N$, we know that for all $n\in\N$ and $t\in\R$,
\[
\|u_n(t)\|_{\dot{H}^\half(\C_+)}
	\leq \|u_0\|_{\dot{H}^\half(\C_+)}.
\]
Using the equation satisfied by $u_n$, we also know that
\[
\|\partial_tu_n(t)\|_{\dot{H}^{-\half}(\C_+)}
	\leq \|P_0^n(|u_n|^2u_n)(t)\|_{\dot{H}^{-\half}(\C_+)}.
\]
By the dual Sobolev embedding $L^{\frac{4}{3}}(\C_+)\hookrightarrow \dot{H}^{-\half}(\C_+)$ and the fact that $P_0$ extends to a bounded projector from $L^p(\C_+)$ to itself as soon as $1<p<+\infty$ (see for example \cite{bergman_projectors}), we can estimate
\begin{align*}
\|P_0^n(|u_n|^2u_n)(t)\|_{\dot{H}^{-\half}(\C_+)}
	&\leq \|P_0(|u_n|^2u_n)(t)\|_{\dot{H}^{-\half}(\C_+)}\\
	&\leq C\|P_0(|u_n|^2u_n)(t)\|_{L^{\frac{4}{3}}(\C_+)}\\
	&\leq C'\||u_n|^2u_n(t)\|_{L^{\frac{4}{3}}(\C_+)}\\
	&\leq C'\|u_n(t)\|_{L^{4}(\C_+)}^3.
\end{align*}
Since $u_n(t)$ is uniformly bounded in $\dot{H}^\half(\C_+)$ and therefore in $L^4(\C_+)$, we conclude that the term $\|\partial_t u_n(t)\|_{\dot{H}^{-\half}(\C_+)}$ is also uniformly bounded.

We now prove that that up to a subsequence, $(u_n)_n$ converges in $\classeC([-T,T],\dot{H}^\half(\C_+)\cap\holo(\C_+))$ (with the weak topology) to a function $u$ for all $T>0$.

We know that $\dot{H}^{-\half}(\C_+)\cap\holo(\C_+)$ is separable, since it is isometric to $L^2(\R_+)$. Moreover, by cutting the Fourier function $f$ at infinity, one can see that $\dot{H}^{-\half}(\C_+)\cap\dot{H}^\half(\C_+)\cap\holo(\C_+)$ is dense in $\dot{H}^{-\half}(\C_+)\cap\holo(\C_+)$. We can therefore consider a countable sequence $(\varphi_k)_k$ in  $\dot{H}^{-\half}(\C_+)\cap\dot{H}^\half(\C_+)\cap\holo(\C_+)$ such that every function in $\dot{H}^{-\half}(\C_+)\cap\holo(\C_+)$ can be approximated by a subsequence of $(\varphi_k)_k$ for the $\dot{H}^{-\half}$-norm.

Fix $k\in\N$. Since $(t\mapsto\partial_t u_n(t))_n$ and $(t\mapsto u_n(t))_n$ are uniformly bounded in $\dot{H}^{-\half}(\C_+)$ and in $\dot{H}^{\half}(\C_+)$ respectively, the sequence $\ell_n(\cdot,\varphi_k):t\in[-T,T]\mapsto (u_n(t),\varphi_k)$ is equicontinuous and equibounded~: for all $n$ and $t$,
\[
|\partial_t\ell_n(t,\varphi_k)|
	=|(\partial_t u_n(t),\varphi_k)|
	\leq \|\partial_t u_n(t)\|_{\dot{H}^{-\half}(\C_+)}\|\varphi_k\|_{\dot{H}^\half(\C_+)}
\]
and
\[
|\ell_n(t,\varphi_k)|
	=|(u_n(t),\varphi_k)|
	\leq \| u_n(t)\|_{\dot{H}^{\half}(\C_+)}\|\varphi_k\|_{\dot{H}^{-\half}(\C_+)}.
\]

Applying Ascoli's theorem, for every $k\in\N$, there is a subsequence $(n_p)_p$ such that $(\ell_{n_p}(\cdot,\varphi_k))_p$ converges in $\classeC([-T,T],\C)$ to some continuous function $\ell(\cdot,\varphi_k)$ as $p$ tends to $+\infty$. By a diagonal argument, we can use the same subsequence for all $k\in\N$. Using a second diagonal argument on a sequence of times $(T_n)_n$ going to $+\infty$, we can assume that for all $k$, there exists $\ell(\cdot,\varphi_k)\in\classeC(\R,\C)$ such that for all $T>0$, the sequence $(\ell_{n_p}(\cdot,\varphi_k))_p$ converges in $\classeC([-T,T],\C)$ to $\ell(\cdot,\varphi_k)|_{[-T,T]}$. 

By density, $\ell$ extends to a bounded linear map $\ell\in\classeC(\R,(\dot{H}^{-\half}(\C_+)\cap\holo(\C_+))^*)$ (with the weak topology). Now, by duality, $\ell$ can be represented by $u\in\classeC(\R,\dot{H}^\half(\C_+)\cap\holo(\C_+))$~: for all $\varphi\in \dot{H}^{-\half}(\C_+)\cap\holo(\C_+)$,
\[
\ell(t,\varphi)=(u(t),\varphi).
\]
To conclude, by construction, for all $T>0$, the sequence $(\ell_{n_p}|_{[-T,T]})_p$ converges weakly to $\ell|_{[-T,T]}$ in the space $\classeC([-T,T],(\dot{H}^{-\half}(\C_+)\cap\holo(\C_+))^*)$, therefore $(u_n)_n$ converges weakly to $u$ in $\classeC([-T,T],\dot{H}^{\half}(\C_+)\cap\holo(\C_+))$. Passing to the limit, we conclude that $u$ is a global solution to the original equation \eqref{eq:bergman} in the distribution sense.

We deduce that
\begin{align*}
d(u(t),\M)^2
	&=\inf_{X\in\R\times\T\times\R_+^*}\|u(t)-T_XQ\|_{\dot{H}^\half(\C_+)}^2\\
	&\leq \inf_{X\in\R\times\T\times\R_+^*}\liminf_{n\to+\infty}\|u_n(t)-T_XQ\|_{\dot{H}^\half(\C_+)}^2.
\end{align*}
Since $X$ is not compact, this inequality is not sufficient if we want to apply inequality \eqref{ineq:delta(u_n)} to estimate $d(u(t),\M)$. In the following part, we construct a map $t\mapsto X_n(t)$, such that for all $t\in\R$, $u_n(t)$ is close to $T_{X_n(t)}Q$ and $(X_n(t))_{n\in\N}$ stays bounded, then use a compactness argument.

\subsection{Modulation}\label{subsection:limiting_modulation}

Recall the notation. Fix $u\in\hardy=\dot{H}^\half(\C_+)\cap\holo(\C_+)$, $X=(s,\theta,\alpha)\in\R\times\T\times\R_+^*$, we denote by $T_Xu$ the element of $\hardy$ satisfying
\[
T_Xu(z):=\e^{i\theta}\alpha u(\alpha^2(z-s)), \quad z\in\C_+.
\]
We write $X^{-1}=(-s,-\theta,\alpha^{-1})$ and
\[
|X|:=|s|+|\theta|+|\log(\alpha)|.
\]
We have also defined
\[
d(u,\M)=\inf_{X=(s,\theta,\alpha)\in \R\times\T\times\R_+^*}\|T_Xu-Q\|_{\dot{H}^\half(\C_+)},
\]
as the distance of $u$ to the orbit of $Q$
\[
\M=\{T_XQ\mid X\in \R\times\T\times\R_+^*\}.
\]

We choose $0<r<1$, and assume that $\|u_0-Q\|_{\dot{H}^{\half}(\C_+)}<r^2$. For $n\geq N$ large enough and $K>0$, the regularized initial data $u_0^n$ satisfies $\delta(u_0^n)<Kr^2$. Applying Proposition \ref{prop:estimates_stability}, there exist $c_0>0$ and $r_0>0$ such that if $0<r<r_0$, then $d(u_n(t),\M)<c_0r$ for all $t\in\R$.

We start from the observation that around time $t=0$, one can choose $X_n(t)=(0,0,1)$ for all $n\geq N$ since $\|u_0^n-Q\|_{\dot{H}^\half(\C_+)}<c_0r$. By continuity, we know that $\|u_n(t)-Q\|_{\dot{H}^\half(\C_+)}\leq (1+\varepsilon)c_0r$ on some small time interval, which can be taken independently of $n$. Indeed,
\begin{align*}
\|u_n(t)-Q\|_{\dot{H}^\half(\C_+)}^2
	&=\|u_n(t)\|_{\dot{H}^\half(\C_+)}^2+\|Q\|_{\dot{H}^\half(\C_+)}^2-2(u_n(t),Q)_{\dot{H}^\half(\C_+)}\\
	&=\|u_0^n\|_{\dot{H}^\half(\C_+)}^2+\|Q\|_{\dot{H}^\half(\C_+)}^2-2(u_n(t),-iQ_z)_{L^2(\C_+)},
\end{align*}
therefore the derivative of $\|u_n(t)-Q\|_{\dot{H}^\half(\C_+)}^2
$ is bounded by
\begin{align*}
\left|\frac{\d}{\d t}\|u_n(t)-Q\|_{\dot{H}^\half(\C_+)}^2\right|
	&=\left|2(\partial_t u_n(t),-iQ_z)_{L^2(\C_+)}\right|\\
	&\leq 2\|\partial_t u_n(t)\|_{\dot{H}^{-\half}(\C_+)}\|-iQ_z\|_{\dot{H}^\half(\C_+)}.
\end{align*}
But we have already seen that $\|\partial_t u_n(t)\|_{\dot{H}^{-\half}(\C_+)}$ is bounded independently of $t$ and $n$, therefore there exists $K>0$ such that for $n\geq N$ and $t\in\R$
\begin{align*}
\|u_n(t)-Q\|_{\dot{H}^{\half}(\C_+)}^2
	&\leq \|u_0^n-Q\|_{\dot{H}^{\half}(\C_+)}^2+K|t|\\
	&\leq (c_0r)^2+K|t|.
\end{align*}
For fixed $\varepsilon>0$, we conclude that $\|u_n(t)-Q\|_{\dot{H}^{\half}(\C_+)}\leq (1+\varepsilon) c_0r$ as long as $|t|\leq \frac{(1+\varepsilon)^2-1}{K}(c_0r)^2$.

Set $\varepsilon>0$ and $t_1:=\frac{(1+\varepsilon)^2-1}{K}(c_0r)^2$. Assume that at time $t_0$, there exists a bounded sequence $(X_n^0)_n$ in $\R\times\T\times\R_+^*$ such that for all $n$, $\| u_n(t_0)-T_{X_n^0}Q\|_{\dot{H}^{\half}(\C_+)}<c_0r$. By the above method, one can show that $\| u_n(t)-T_{X_n^0}Q\|_{\dot{H}^{\half}(\C_+)}\leq (1+\varepsilon)c_0r$ on $[t_0-t_1,t_0+t_1]$. Indeed, let $v_n:=T_{(X_n^0)^{-1}} u_n$. The equation satisfied by $u_n$ is not invariant by scaling, but one can explicit which equation is satisfied by $v_n$.  Recall that if
\[
u(z)=\frac{1}{\sqrt{2\pi}}\int_0^{+\infty}e^{iz\sigma}f(\sigma)\d\sigma,
\]
then
\[
\widetilde{P}_{\varepsilon,M} u(z)=\frac{1}{\sqrt{2\pi}}\int_{\varepsilon}^Me^{iz\sigma}f(\sigma)\d\sigma.
\]
Write $(X_n^0)=:(s_n^0,\theta_n^0,\alpha_n^0)$ and $\widetilde{P_0^n}:=\widetilde{P}_{\frac{\varepsilon_n}{(\alpha_n^0)^2},\frac{1}{\varepsilon_n(\alpha_n^0)^2}}\circ P_0$, then $v_n=T_{(X_n^0)^{-1}}u_n$ satisfies
\[
i(v_n)_t=\widetilde{P_0^n}(|v_n|^2v_n),
\]
moreover $\|v_n(t_0)-Q\|_{\dot{H}^{\half}(\C_+)}<c_0r$. But we have the same inequalities as above
\begin{align*}
\|\partial_tv_n(t)\|_{\dot{H}^{-\half}(\C_+)}
	&\leq \|\widetilde{P_0}(|v_n|^2v_n)(t)\|_{\dot{H}^{-\half}(\C_+)}\\
	&\leq C\|\widetilde{P_0}(|v_n|^2v_n)(t)\|_{L^{\frac{4}{3}}(\C_+)}\\
	&\leq C'\||v_n|^2v_n(t)\|_{L^{\frac{4}{3}}(\C_+)}\\
	&\leq C'\|v_n(t)\|_{L^{4}(\C_+)}^3.
\end{align*}
Since $\|v_n(t)\|_{L^{4}(\C_+)}=\|u_n(t)\|_{L^{4}(\C_+)}$ is uniformly bounded by conservation of the $L^4$-norm, we conclude that $\|v_n(t)-Q\|_{\dot{H}^{\half}(\C_+)}=\| u_n(t)-T_{X_n^0}Q\|_{\dot{H}^{\half}(\C_+)}\leq (1+\varepsilon) c_0r$ as long as $|t-t_0|\leq t_1$.

We construct $X_n$ as a piecewise $\classeC^1$ functional on $\R$ as follows. For $k\in\Z$, $X_n$ is constant on $[kt_1,(k+1)t_1[$, equal to some $X_n^k\in\R\times\T\times\R_+^*$ to be chosen.  We first set $X_n^{-1}=X_n^0=(0,0,1)$. Then, at time $t_k=kt_1$, $k\geq 1$, we use the fact that $d(u_n(t_k),\M)<r$ and choose $X_n^k$ such that $\|u_n(t_k)-T_{X_n^k}Q\|_{\dot{H}^\half(\C_+)}<c_0r$. Then from the above paragraph, $\|u_n(t)-T_{X_n^k}Q\|_{\dot{H}^\half(\C_+)}\leq (1+\varepsilon)c_0r$ on $[t_k,t_k+t_1]$. We do a similar construction for negative times. The map $X_n$ satisfies
\[
\|u_n(t)-T_{X_n(t)}Q\|_{\dot{H}^\half(\C_+)}\leq (1+\varepsilon)c_0r,
	\quad  t\in\R.
\]

It remains to show that $X_n$ is bounded independently of $n$ on bounded intervals. In order to do so, it is enough to control the gap between $X_n^{k-1}$ and $X_n^{k}$. By construction, at time $t_{k}$,
\[
\|u_n(t_k)-T_{X_n^{k-1}}Q\|_{\dot{H}^\half(\C_+)}\leq (1+\varepsilon)c_0r
\]
and
\[
\|u_n(t_k)-T_{X_n^k}Q\|_{\dot{H}^\half(\C_+)} < c_0r,
\]
therefore
\[
\|T_{X_n^{k-1}}Q-T_{X_n^k}Q\|_{\dot{H}^\half(\C_+)}\leq (2+\varepsilon)c_0r.
\]
Using the following Lemma, we conclude that if $r$ is chosen small enough, then there exists a constant $c_1>0$ such that for all $n\geq N$ and $k\in\Z$,
\[
|X_n^{k-1}(X_n^k)^{-1}|\leq c_1.
\]

\begin{lem}\label{lem:gap_TXQ_Q} For some constants $c_1,r_1>0$, the following holds. Let $X\in\R\times\T\times\R_+^*$ such that
\[
\|T_XQ-Q\|_{\dot{H}^\half(\C_+)}\leq r_1.
\]
Then
\[
|X|\leq c_1.
\]
\end{lem}

\begin{proof}
Thanks to the invariance of the $\dot{H}^\half$-norm by symmetries, one can assume that $X=(s,\theta,\alpha)$ with $\alpha\geq 1$ up to exchanging $X$ and $X^{-1}$. We develop
\begin{align*}
\|T_XQ-Q\|_{\dot{H}^\half(\C_+)}^2
	&=\|T_XQ\|_{\dot{H}^\half(\C_+)}^2+\|Q\|_{\dot{H}^\half(\C_+)}^2-2(T_XQ,Q)_{\dot{H}^\half(\C_+)}\\
	&=2\pi-2(T_XQ,Q)_{\dot{H}^\half(\C_+)}.
\end{align*}
Now, recall that
\[
Q(z)
	=\frac{\sqrt{2}}{z+i}
	=\frac{1}{\sqrt{2\pi}}\int_0^{+\infty}e^{iz\sigma}f(\sigma)\d\sigma
\]
with
\[
f(\sigma)=-2i\sqrt{\pi}e^{-\sigma},
\]
and
\[
\|Q\|_{\dot{H}^\half(\C_+)}^2
	=\half\int_0^{+\infty}|f(\sigma)|^2\d\sigma
	=\pi.
\]
With this notation, the function corresponding to $T_XQ$ is $g(\sigma)=-2i\sqrt{\pi}e^{i\theta}e^{-is\sigma}e^{-\frac{\sigma}{\alpha^2}}\frac{1}{\alpha^2}$, therefore
\begin{align*}
\|T_XQ-Q\|_{\dot{H}^\half(\C_+)}^2
	&=2\pi-4\pi\Re\left(\int_0^{+\infty}e^{i\theta}e^{-is\sigma}e^{-\frac{\sigma}{\alpha^2}}\frac{1}{\alpha^2}e^{-\sigma}\d\sigma\right)\\
	&=2\pi-4\pi\Re\left(\frac{e^{i\theta}}{\alpha^2\left(is+\frac{1}{\alpha^2}+1\right)}\right).
\end{align*}
Set $\alpha=1+\beta$ with $\beta\geq 0$. We want to bound $s$ and $\beta$. By assumption,
\[
\left|\Re\left(\frac{e^{i\theta}}{is\frac{(1+\beta)^2}{2}+1+\beta+\frac{\beta^2}{2}}\right)-1\right|
	\leq \frac{r_1^2}{2\pi}=:\delta_1.
\]
Denote $z:=\frac{e^{i\theta}}{is\frac{(1+\beta)^2}{2}+1+\beta+\frac{\beta^2}{2}}$. The fact $|\Re(z)-1|\leq \delta_1$ implies that $|z|\geq \Re(z)\geq 1-\delta_1$, and if $\delta_1<1$, that
\[
\frac{1}{|z|}
	=\left|is\frac{(1+\beta)^2}{2}+1+\beta+\frac{\beta^2}{2}\right|
	\leq \frac{1}{1-\delta_1}.
\]
On the one hand, taking the real part,
\[
1+\beta+\frac{\beta^2}{2}\leq \frac{1}{1-\delta_1}.
\]
Since $\beta\in\R_+\mapsto 1+\beta+\frac{\beta^2}{2}$ is strictly increasing and going to $+\infty$ as $\beta$ goes to $+\infty$, there exists some constant $c>0$ such that $\beta\leq c$, or in other terms $0\leq \log\alpha\leq \log(1+c)$. On the other hand, since $\beta\geq 0$, the bound on the imaginary part implies that
\[
|s|\leq\frac{2}{1-\delta_1}.
\]
\end{proof}

Using the Lemma, assume that $3c_0r<r_1$ and fix $t\in\R$. We now know that $(X_n(t))_n$ takes values in a compact set : up to extraction, one can assume that $(X_n(t))_n$ converges to some $X(t)$. Moreover, for all $t\in\R$ and $n\in\N$, $\|u_n(t)-T_{X_n(t)}Q\|_{\dot{H}^\half(\C_+)}\leq (1+\varepsilon)c_0r$, therefore passing to the weak limit $n\to+\infty$ we conclude that $\|u(t)-T_{X(t)}Q\|_{\dot{H}^\half(\C_+)}\leq (1+\varepsilon) c_0r$. Since $\varepsilon>0$ can be taken arbitrarily small, we have proven the following reformulation of Theorem~\ref{thm:stability_Q}.

\begin{thm} For some $c_0>0$ and $r_0>0$, the following holds. Let $r\leq r_0$ and $u_0\in\dot{H}^\half(\C_+)\cap\holo(\C_+)$ such that $\|u_0-Q\|_{\dot{H}^{\half}(\C_+)}< r^2$. Then there exists a weak solution $u\in\classeC(\R,\dot{H}^\half(\C_+)\cap\holo(\C_+))$ (with the weak topology) to equation \eqref{eq:bergman}
\begin{equation*}
\begin{cases}
i\partial_t u=P_0(|u|^2u)
\\
u(t=0)=u_0
\end{cases}, \quad (t,z)\in\R\times \C_+
\end{equation*}
such that for all $t\in\R$,
\[
d(u(t),\M)\leq c_0r.
\]
\end{thm}

\section{Conditional orbital stability of the ground states \texorpdfstring{$Q_\beta$}{Q_beta} in the Schrödinger equation}\label{section:stability_Qbeta}

We now consider the Schrödinger equation on the Heisenberg group \eqref{eq:H}
\begin{equation*}
\begin{cases}
i\partial_t u-\hlapl u=|u|^2u
\\
u(t=0)=u_0
\end{cases}, \quad (t,x,y,s)\in\R\times\heis.
\end{equation*}
For $\beta\in(\beta_*,1)$, we are interested in solutions with initial data $u_0\in\dot{H}^1(\heis)$ satisfying
\[
\|u_0-\sqrt{1-\beta}Q_\beta\|_{\dot{H}^1(\heis)}<(1-\beta)r.
\]

Let $u$ be an eventual solution, and set
\[
u(t,x,y,s)=\sqrt{1-\beta}U((1-\beta)t,x,y,s+\beta t),
\]
so that $U$ is a solution to
\begin{equation}\label{eq:H_rescaled}
i\partial_t U-\frac{\hlapl+\beta D_s}{1-\beta} U=|U|^2U, \quad (t,x,y,s)\in\R\times\heis.
\end{equation}
The initial data $U_0$ satisfies
\[
\|U_0-Q_\beta\|_{\dot{H}^1(\heis)}<\sqrt{1-\beta}r.
\]

There are two relevant conserved quantities for this equation : the energy
\[
\mathcal{E}_\beta(V):=\half(-\frac{\hlapl+\beta D_s}{1-\beta} V,V)_{L^2(\heis)}-\frac{1}{4}\|V\|_{L^4(\heis)}^4,
\]
and the momentum
\[
\mathcal{P}(V):=(D_s V,V)_{L^2(\heis)},
	\quad V\in\dot{H}^1(\heis).
\]

Theorem \ref{thm:stability_Qbeta} is equivalent to prove that if $\beta$ is large, then one can construct a weak global solution $U$ to equation \eqref{eq:H_rescaled} which stays close to the orbit of $Q_\beta$ for all times, which leads to the following reformulation.
\begin{thm}\label{thm:stability_Qbeta_precise}
For some constants $c_0\in\R_+^*$ and $r_0\in\R_+^*$, for all $r\in(0,r_0)$, there exists $\beta^*(r)\in(0,1)$ such that the following holds. Let $\beta\in(\beta^*(r),1)$ and $U_0\in\dot{H}^1(\heis)$ satisfying
\begin{itemize}
\item if $U_0\in\dot{H}^1(\heis)\cap V_0^+$ :
\[
\|U_0-Q_\beta\|_{\dot{H}^1(\heis)}< r^2
\]
\item in the general case :
\[\|U_0-Q_\beta\|_{\dot{H}^1(\heis)}<\sqrt{1-\beta}r.
\]
\end{itemize}

Then there exists a global weak solution $U_\beta\in\classeC(\R,\dot{H}^1(\heis))$ (with the weak topology) to equation~\eqref{eq:H_rescaled}
\begin{equation*}
\begin{cases}
i\partial_t U_{\beta} -\frac{\hlapl+\beta D_s}{1-\beta} U_{\beta}=|U_{\beta}|^2U_{\beta}
\\
U_\beta(t=0)=U_0
\end{cases}
\end{equation*}
such that for all $t\in\R$, $U_\beta(t)$ is close to the orbit 
$\Qbeta=\{T_XQ_\beta \mid X\in\R\times\T\times\R_+^*\}$ of $Q_\beta$~:
\[
d(U_\beta(t),\Qbeta)\leq c_0r.
\]
\end{thm}

Comparing to the strategy deployed for the half-wave equation, the gap
\[
\delta_\beta(V)
	:=\left|\mathcal{E}_\beta(V)-\mathcal{E}_\beta(Q_\beta)\right|+\left|\mathcal{P}(V)-\mathcal{P}(Q_\beta)\right|,
	\quad V\in\dot{H}^1(\heis),
\]
does not here directly control the distance of $V$ to $\Qbeta$. Indeed, even the fact that $\delta_\beta(V)=0$ does not imply that $V$ belongs to $\Qbeta$. This is due to the fact that we can only use two conservation laws (energy and momentum) here, whereas an additional conservation laws was available for the half-wave equation : the mass of the solution.

However, using that $Q_\beta$ tends to $Q$ as $\beta$ tends to $1$,  one can instead show that the component of the solution along the space $V_0^+$ is close to $Q$ and control the rest separately. More precisely, decompose
\[U(t)=U^+(t)+W(t),
\]
where $U^+(t)\in \dot{H}^1(\heis)\cap V_0^+$ and $W(t)\in\bigoplus_{(k,\pm)\neq(0,+)}\dot{H}^1(\heis)\cap V_k^\pm$. If we know that $W(t)$ is small enough, then $\delta_\beta(U(t))\approx\delta(U^+(t))$. This enables us to estimate the distance $d(U^+(t),\M)$ of $U^+(t)$ to the orbit of $Q$, and therefore the distance of $U(t)$ to the orbit of $Q_\beta$ for $\beta$ close to $1$.

The plan of the proof is is as follows. Fix $\beta\in(0,1)$. We approximate the initial data and the equation by smooth global functions $(U_{\gamma,n})_{\gamma\in[\beta,1),n\in\N}$ valued in $H^2(\heis)$ in part \ref{subsection:approximate_solutions_qbeta}. We then decompose
\[
U_{\gamma,n}(t)=U_{\gamma,n}^+(t)+W_{\beta,n}(t),
\]
where $U_{\gamma,n}^+(t)\in \dot{H}^1(\heis)\cap  V_0^+$ and $W_{\gamma,n}(t)\in\bigoplus_{(k,\pm)\neq(0,+)}\dot{H}^1(\heis)\cap V_k^\pm$. In part \ref{subsection:limit_beta_n_fixed}, we fix $n\in\N$, and study the limit $\gamma\to 1$. We prove by using the conservation laws that $W_{\gamma,n}(t)$ stays small and that the gap $\delta(U_{\gamma,n}^+(t))$ is controlled as $\delta(U_{\gamma,n}^+(t))\lesssim r^2$ for $\gamma\geq \beta^*(n,t)$, which leads to an upper bound
\begin{equation}\label{ineq:distUM}
d(U_{\gamma,n}(t),\M)< c_0r,
	\quad t\in\R, \gamma\in[\max(\beta^*(n,t),\beta),1).
\end{equation}
Then, we show that the lower bound $\beta^*(n,t)$ can be taken independently of $n$ and $t$. Finally, in part \ref{subsection:limit_n_beta_fixed}, we fix $\beta\geq\beta_*$ and use the same method as for the limiting equation to find an upper bound on the modulation parameters $(X_{\beta,n}(t))_{n\in\N}$ in order to pass to the limit $n\to+\infty$ in the above inequality \eqref{ineq:distUM}.

\subsection{Construction of approximate solutions}\label{subsection:approximate_solutions_qbeta}

\paragraph{Construction of a sequence of smoothing projectors $\Pi^{(n)}$ :}

We define a sequence of projectors $\Pi^{(n)}$ close to identity, mapping elements of $\dot{H}^s(\heis)$ ($s=\pm 1$) to smoother functions, by cutting frequencies in the decomposition
\[
\dot{H}^{s}(\heis)=\bigoplus_{k\in\N}\bigoplus_{\pm}\dot{H}^{s}(\heis)\cap V_k^\pm.
\]
Using these projectors, we consider a sequence of equations approximating \eqref{eq:H_rescaled} for which the Cauchy problem is globally well-posed.

Let $u\in \dot{H}^{s}(\heis)$, $s=\pm 1$, which we decompose as a series of elements of $\dot{H}^s(\heis)\cap V_k^{\pm}$ for $(k,\pm)\in(\N,\pm)$. Write
\[
u=\sum_{k=0}^{+\infty}\sum_\pm\Pi_{k,\pm}(u),
\]
where for all $(k,\pm)\in(\N,\pm)$, $\Pi_{k,\pm}(u)\in \dot{H}^s(\heis)\cap V_k^{\pm}$. Then
\[
\|u\|_{\dot{H}^s(\heis)}^2
	=\sum_{k\in\N}\sum_\pm\int_{\R_{\pm}}((k+1)|\sigma|)^s\int_{\R^2}|\widehat{\Pi_{k,\pm}(u)}(x,y,\sigma)|^2\d x\d y\d\sigma.
\]

Let $n\in\N$, we define $\Pi^{(n)}(u)$ as follows.  We take the $n$-th partial sum and cut off the frequencies $|\sigma|\to 0$ and $|\sigma|\to+\infty$ :
\begin{align*}
\widehat{\Pi^{(n)}(u)}(x,y,\sigma)
	&:=\sum_{k=0}^{n}\sum_\pm\widehat{\Pi_{k,\pm}(u)}(x,y,\sigma)\un_{\frac{1}{n}\leq |\sigma|\leq n}.
\end{align*}
Consequently,
\[
\|\Pi^{(n)}(u)\|_{\dot{H}^s(\heis)}^2
	=\sum_{k=0}^n\sum_{\pm}\int_{\{\sigma\in\R_{\pm}\}\cap\{\frac{1}{n}\leq |\sigma|\leq n\}}((k+1)|\sigma|)^s\int_{\R^2}|\widehat{\Pi_{k,\pm}(u)}(x,y,\sigma)|^2\d x\d y\d\sigma
\]
converges to $\|u\|_{\dot{H}^s(\heis)}^2$ as $n$ goes to $+\infty$.

Moreover, if $u\in\dot{H}^1(\heis)$, then $\Pi^{(n)}(u)$ belongs to $H^2(\heis)$. Indeed,
\begin{align*}
\|\Pi^{(n)}(u)\|_{H^2(\heis)}^2
	&=\sum_{k=0}^n\sum_{\pm}\int_{\{\sigma\in\R_{\pm}\}\cap\{\frac{1}{n}\leq |\sigma|\leq n\}}(1+(k+1)^2|\sigma|^2)\int_{\R^2}|\widehat{\Pi_{k,\pm}(u)}(x,y,\sigma)|^2\d x\d y\d\sigma,
\end{align*}
but on the set $\{\frac{1}{n}\leq |\sigma|\leq n\}$,
\[
(1+(k+1)^2|\sigma|^2)
	\leq n(n+2)(k+1)|\sigma|,
\]
therefore $\|\Pi^{(n)}(u)\|_{H^2(\heis)}$ is finite.

\paragraph{Construction of a sequence of approximate solutions $(U_{\gamma,n})_{\gamma\in[\beta,1),n\in\N}$ :}

Fix $\beta\in(0,1)$, $r>0$ and $U_0\in\dot{H}^1(\heis)$ such that
\[
\|U_0-Q_\beta\|_{\dot{H}^1(\heis)}<\sqrt{1-\beta}r.
\]

We want to construct a global solution to \eqref{eq:H_rescaled}
\begin{equation*}
\begin{cases}
i\partial_t U_{\beta} -\frac{\hlapl+\beta D_s}{1-\beta} U_{\beta}=|U_{\beta}|^2U_{\beta}
\\
U_\beta(t=0)=U_0
\end{cases}
\end{equation*}
such that for all $t\in\R$,
\[
d(U_\beta(t),\Qbeta)\leq c_0r.
\]
By approximation, the idea would be to consider a sequence of equations
\begin{equation}\label{eq:H_rescaled_n}
\begin{cases}
i\partial_t U_{\beta,n} -\frac{\hlapl+\beta D_s}{1-\beta} U_{\beta,n}=\Pi^{(n)}(|U_{\beta,n}|^2U_{\beta,n})
\\
U_{\beta,n}(t=0)=U_0^{\beta,n}=\Pi^{(n)}(U_0)
\end{cases},
\quad n\in\N,
\end{equation}
for which one can show that for all $n$ large, there exists $\beta^*(n)$ such that if  $\beta\geq\beta^*(n)$, then
\[
d(U_{\beta,n}(t),\Qbeta)\leq c_0r, \quad t\in\R.
\]

In order to get a lower bound $\beta^*(n)$ independent of $n$, we rather construct a set of initial data $(U_0^{\gamma,n})_{\gamma\in[\beta,1),n\in\N}$ and equations
\begin{equation}\label{eq:H_rescaled_n_beta}
\begin{cases}
i\partial_t U_{\gamma,n} -\frac{\hlapl+\gamma D_s}{1-\gamma} U_{\gamma,n}=\Pi^{(n)}(|U_{\gamma,n}|^2U_{\gamma,n})
\\
U_{\gamma,n}(t=0)=U_0^{\gamma,n}=\Pi^{(n)}(U_0^\gamma)
\end{cases},
\quad n\in\N,\gamma\in[\beta,1),
\end{equation}
then use a continuity argument.

For $\gamma\in[\beta,1)$, the initial data $U_0^{\gamma}$ is defined as follows :
\[
U_0^\gamma:=\frac{1-\gamma}{1-\beta}U_0+\frac{\gamma-\beta}{1-\beta} Q.
\]
If $U_0\in \dot{H}^1(\heis)\cap V_0^+$, we choose $U_0^\gamma$ constant equal to $U_0$. However, in the general case, we cannot use this choice because we need the initial data $U_0^\gamma$ to go to $\dot{H}^1(\heis)\cap V_0^+$ as $\gamma$ tends to $1$.

\begin{lem}
For some constants $C_0>0$, $\beta_*(r)\in(0,1)$ and $N(r)\in \N$ the following holds. Assume that $\beta\in(\beta_*(r),1)$. Then for all $n\geq N(r)$ and $\gamma\in[\beta,1)$,
\[
\left|\mathcal{E}_{\gamma}(U_0^{\gamma,n})-\mathcal{E}(Q)\right|+\left|\mathcal{P}(U_0^{\gamma,n})-\mathcal{P}(Q)\right|
	<C_0r^2.
\]
\end{lem}

\begin{proof}
We use the following convergence rate of $(Q_\beta)_\beta$ to $Q$ as $\beta$ tends to $1$ (proved in Appendix~\ref{appendix})~: 
\[
\|Q_\beta-Q\|_{\dot{H}^1(\heis)}=o(\sqrt{1-\beta}).
\]

If $U_0\in\dot{H}^1(\heis)\cap V_0^+$ and $\|U_0-Q_\beta\|_{\dot{H}^1(\heis)}<r^2$, we have chosen $U_0^\gamma$ constant equal to $U_0$ and it is enough to use that $\|\Pi^{(n)}(U_0)-U_0\|_{\dot{H}^1(\heis)}\to 0$ as $n\to+\infty$.

We now treat the case $\|U_0-Q_\beta\|_{\dot{H}^1(\heis)}<\sqrt{1-\beta}r$. By convergence of $Q_\beta$ to $Q$, there exists $\beta_*=\beta_*(r)\in(0,1)$ such that for all $\beta\in(\beta_*,1)$,
\[
\|Q_\beta-Q\|_{\dot{H}^1(\heis)}<\sqrt{1-\beta}r.
\]

We decompose
\[
Q_\beta=Q_\beta^++R_\beta
\]
and
\[
U_0=U_0^++W_0
\]
where $Q_\beta^+,U_0^+\in\dot{H}^1(\heis)\cap V_0^+$ and $R_\beta,W_0\in \bigoplus_{(k,\pm)\neq(0,+)}\dot{H}^1(\heis)\cap V_k^\pm$. In the same way, we decompose $U_0^\gamma$ as
\[
U_0^{\gamma}=(U_0^{\gamma})^++W_0^{\gamma}
\]
and $U_0^{\gamma,n}=\Pi^{(n)}(U_0^\gamma)$ as
\[
U_0^{\gamma,n}=(U_0^{\gamma,n})^++W_0^{\gamma,n},
\]
where $(U_0^{\gamma})^+,(U_0^{\gamma,n})^+\in\dot{H}^1(\heis)\cap V_0^+$ and $W_0^{\gamma},W_0^{\gamma,n}\in \bigoplus_{(k,\pm)\neq(0,+)}\dot{H}^1(\heis)\cap V_k^\pm$.

Since
\[
\|W_0-R_\beta\|_{\dot{H}^1(\heis)}
	\leq\|U_0-Q_\beta\|_{\dot{H}^1(\heis)}
	<\sqrt{1-\beta}r,
\]
then $W_0$ satisfies
\begin{align*}
\|W_0\|_{\dot{H}^1(\heis)}
	&\leq \|W_0-R_\beta\|_{\dot{H}^1(\heis)}+\|R_\beta\|_{\dot{H}^1(\heis)}\\
	&\leq 2\sqrt{1-\beta}r.
\end{align*}
Therefore, $W_0^\gamma=\frac{1-\gamma}{1-\beta}W_0$ satisfies
\[
\|W_0^{\gamma}\|_{\dot{H}^1(\heis)}\leq 2\frac{1-\gamma}{\sqrt{1-\beta}}r,
\]
which implies that for all $n\in\N$,
\[
\|W_0^{\gamma,n}\|_{\dot{H}^1(\heis)}\leq 2\frac{1-\gamma}{\sqrt{1-\beta}}r.
\]
In particular,
\begin{align*}
\left|(-\frac{\hlapl+\gamma D_s}{1-\gamma}W_0^{\gamma,n},W_0^{\gamma,n})_{L^2(\heis)}\right|+\left|(D_sW_0^{\gamma,n},W_0^{\gamma,n})_{L^2(\heis)}\right|
	&\leq 8\frac{1-\gamma}{1-\beta}r^2+4\frac{(1-\gamma)^2}{1-\beta}r^2\\
	&\leq 12r^2.
\end{align*}

Given the form of the energy
\begin{multline*}
\mathcal{E}_\gamma(U_0^{\gamma,n})
	=\half (-\frac{\hlapl+\gamma D_s}{1-\gamma}W_0^{\gamma,n},W_0^{\gamma,n})_{L^2(\heis)}+\half(D_s (U_0^{\gamma,n})^+,(U_0^{\gamma,n})^+)_{L^2(\heis)}
	\\-\frac{1}{4}\|(U_0^{\gamma,n})^++W_0^{\gamma,n}\|_{L^4(\heis)}^4,
\end{multline*}
it is now enough to estimate $\|(U_0^{\gamma,n})^+-Q\|_{\dot{H}^1(\heis)}$. But,
\begin{align*}
\|(U_0^{\gamma,n})^+-(U_0^{\gamma})^+\|_{\dot{H}^1(\heis)}
	&\leq\frac{1-\gamma}{1-\beta}\|\Pi^{(n)}((U_0)^+)-U_0^+\|_{\dot{H}^1(\heis)}
	+\frac{\gamma-\beta}{1-\beta}\|\Pi^{(n)}(Q)-Q\|_{\dot{H}^1(\heis)}\\
	&\leq\|\Pi^{(n)}((U_0)^+)-U_0^+\|_{\dot{H}^1(\heis)}+\|\Pi^{(n)}(Q)-Q\|_{\dot{H}^1(\heis)},
\end{align*}
which converges to zero as $n$ goes to $+\infty$ independently of $\gamma$. Moreover,
\begin{align*}
\|(U_0^\gamma)^+-Q\|_{\dot{H}^1(\heis)}
	&=\frac{1-\gamma}{1-\beta}\|U_0^+-Q\|_{\dot{H}^1(\heis)}\\
	&\leq \frac{1-\gamma}{1-\beta}(\|U_0-Q_\beta\|_{\dot{H}^1(\heis)}+\|W_0\|_{\dot{H}^1(\heis)}+\|Q_\beta-Q\|_{\dot{H}^1(\heis)})\\
	&\leq \frac{1-\gamma}{1-\beta}4\sqrt{1-\beta}r\\
	&\leq 4r^2
\end{align*}
for $\beta\geq\beta_*(r)$ large enough.

To conclude, there exists $C_0>0$, $r_0>0$ and $N\in\N$ such that for all $n\geq N$, $r\in(0,r_0)$ and $\gamma\in[\beta,1)$,
\[
\left|\mathcal{E}_{\gamma}(U_0^{\gamma,n})-\mathcal{E}(Q)\right|+\left|\mathcal{P}(U_0^{\gamma,n})-\mathcal{P}(Q)\right|
	<C_0r^2.
\]
\end{proof}

From now on, we assume that $\beta\geq\beta_*(r)$ and $n\geq N(r)$.

As in Proposition \ref{prop:well_posedness_limit} for the limiting equation, equation \eqref{eq:H_rescaled_n_beta} admits a unique global solution in $H^2_n(\heis):=\Pi^{(n)}(H^2(\heis))$. 

\begin{prop}
Let $U_0^{\gamma,n}\in H^2_{n}(\heis)$. Then there exists a unique $U_{\gamma,n}\in\classeC^\infty(\R,H^2_n(\heis))$ such that \eqref{eq:H_rescaled_n_beta} is satisfied in the distribution sense.
\end{prop}

Here again, in order to show that the local maximal solutions are global, we use the conservation of the momentum 
\[
\mathcal{P}(V)=(D_s V,V)_{L^2(\heis)},
\]
and the following inequality valid for $V\in H^2_n(\heis)$ :
\[
(D_s V,V)_{L^2(\heis)}\leq \|V\|_{H^2(\heis)}^2\leq n(n^2+2n+2)(D_s V,V)_{L^2(\heis)}.
\]

\subsection{Limit \texorpdfstring{$\gamma\to 1$}{gamma to 1} for the \texorpdfstring{$n$}{n}-th partial sum}\label{subsection:limit_beta_n_fixed}

In this part, we use the conservation of energy and momentum to recover an upper bound on $d(U_{\gamma,n}(t),\M)$ for $\gamma\geq\beta^*(n,t)$ close to $1$. Then, we prove that the lower bound for $\gamma$ can be chosen independently of $n$ and $t$.

For $t\in\R$, we decompose $U_{\gamma,n}(t)$ as
\[
U_{\gamma,n}(t)=U_{\gamma,n}^+(t)+W_{\gamma,n}(t).
\]
We show that $U_{\gamma,n}^+(t)\in \dot{H}^1(\heis)\cap V_0^+$ is the main part for which we control $\delta(U_{\gamma,n}^+(t))$, and $W_{\gamma,n}(t)\in\bigoplus_{(k,\pm)\neq (0,+)}\dot{H}^1(\heis)\cap V_k^{\pm}$ is a remainder term which vanishes in the limit $\gamma\to 1$.

First, since $\mathcal{P}(U_{\gamma,n}(t))=(D_sU_{\gamma,n}(t),U_{\gamma,n}(t))_{L^2(\heis)}$ is conserved, bounded by $C_0r^2+\mathcal{P}(Q)$ for all $\gamma\in[\beta,1)$ and equivalent to $\|U_{\gamma,n}(t)\|_{\dot{H}^1(\heis)}^2$ in $H^2_n(\heis)$, there exists some constant $C(n)>0$ such that for all $t\in\R$ and $\gamma\in[\beta,1)$,
\[
\|U_{\gamma,n}(t)\|_{\dot{H}^1(\heis)}\leq C(n).
\]
But such a bound on $\|U_{\gamma,n}(t)\|_{\dot{H}^1(\heis)}$ and the conservation of energy imply that $W_{\gamma,n}(t)$ must vanish as $\gamma$ tends to $1$.

\begin{lem}\label{lem:estimates_w_beta,n}
For some constant $C_1>0$ the following holds. Assume that there exists $C>0$ (possibly depending on $n$), $t\in\R$ and $\gamma\in[\beta,1)$ such that $\|U_{\gamma,n}(t)\|_{\dot{H}^1(\heis)}\leq C$. Then
\[
\|W_{\gamma,n}(t)\|_{\dot{H}^1(\heis)}\leq C_1(1+C^2)\sqrt{1-\gamma}.
\]
\end{lem}

\begin{proof}
We use the conservation of energy
\begin{multline*}
\mathcal{E}_{\gamma}(U_{\gamma,n}(t))
	=\half(-\frac{\hlapl+\gamma D_s}{1-\gamma}W_{\gamma,n}(t),W_{\gamma,n}(t))_{L^2(\heis)}+\half(D_sU_{\gamma,n}^+(t),U_{\gamma,n}^+(t))_{L^2(\heis)}\\-\frac{1}{4}\|U_{\gamma,n}(t)\|_{L^4(\heis)}^4
\end{multline*}
and the fact that
\[
|\mathcal{E}_\gamma(U_0^{\gamma,n})-\mathcal{E}(Q)|<C_0r^2.
\]
Thanks to the embedding $\dot{H}^1(\heis)\hookrightarrow L^4(\heis)$, we know that
\[
\|U_{\gamma,n}(t)\|_{L^4(\heis)}^4\leq K^4\|U_{\gamma,n}(t)\|_{\dot{H}^1(\heis)}^4=K^4C^4.
\]
Moreover, recall the equivalence of norms
\[
\half \|w\|_{\dot{H}^1(\heis)}^2
	\leq (-(\hlapl+\gamma D_s)w,w)_{L^2(\heis)}
	\leq 2 \|w\|_{\dot{H}^1(\heis)}^2,
	\quad w\in\bigoplus_{(k,\pm)\neq (0,+)}\dot{H}^1(\heis)\cap V_k^{\pm}.
\]
We conclude that
\[
\frac{1}{4(1-\gamma)}\|W_{\gamma,n}(t)\|_{\dot{H}^1(\heis)}^2
	\leq \mathcal{E}(Q)+C_0r^2+\frac{1}{4}K^4C^4.
\]
\end{proof}

This Lemma implies that for all $t\in\R$ and $\gamma\in[\beta,1)$,
\[
\|W_{\gamma,n}(t)\|_{\dot{H}^1(\heis)}\leq C_1(1+C(n)^2)\sqrt{1-\gamma},
\] which vanishes as $\gamma$ tends to $1$.

Fix $t\in\R$. Since $\gamma\in[\beta,1)\mapsto W_{\gamma,n}(t)$ is continuous, we can define $\beta_0(n,t)\geq \beta$ as the minimal element in $[\beta,1)$ such that for all $\gamma\in[\beta_0(n,t),1)$,
\[
\|W_{\gamma,n}(t)\|_{\dot{H}^1(\heis)}\leq r^2.
\]

\begin{lem}
For $t\in\R$, $\gamma\in[\beta,1)\mapsto W_{\gamma,n}(t)\in\dot{H}^1(\heis)$ is continuous.
\end{lem}

\begin{proof}
Fix $t\in\R$. We show that $\gamma\in[\beta,1)\mapsto U_{\gamma,n}(t)\in\dot{H}^1(\heis)$ is continuous. Let $\gamma_1,\gamma_2\in[\beta,1)$ and set $R:=U_{\gamma_1,n}-U_{\gamma_2,n}$. Then $R$ is a solution to
\begin{multline*}
i\partial_t R-\frac{\hlapl+\gamma_1 D_s}{1-\gamma_1}R-\left(\frac{\hlapl+\gamma_1 D_s}{1-\gamma_1}-\frac{\hlapl+\gamma_2 D_s}{1-\gamma_2}\right)U_{\gamma_2,n}
	\\=\Pi^{(n)}(|U_{\gamma_1,n}|^2U_{\gamma_1,n})-\Pi^{(n)}(|U_{\gamma_2,n}|^2U_{\gamma_2,n}).
\end{multline*}
We bound $\|\partial_tR(t)\|_{\dot{H}^1(\heis)}$, which is equivalent to controlling $\|\partial_tR(t)\|_{\dot{H}^{-1}(\heis)}$ since $\partial_tR(t)\in\Pi^{(n)}(\dot{H}^1(\heis))$. We treat each term in the equation separately.

First,
\begin{align*}
\|-\frac{\hlapl+\gamma_1 D_s}{1-\gamma_1}R(t)\|_{\dot{H}^{-1}(\heis)}
	&\leq \frac{2}{1-\gamma_1}\|-\hlapl R(t)\|_{\dot{H}^{-1}(\heis)}\\
	&\leq \frac{2}{1-\gamma_1}\| R(t)\|_{\dot{H}^1(\heis)}.
\end{align*}
Then,
\begin{align*}
\|\Big(\frac{\hlapl+\gamma_1 D_s}{1-\gamma_1}-
	&\frac{\hlapl+\gamma_2 D_s}{1-\gamma_2}\Big)U_{\gamma_2,n}(t)\|_{\dot{H}^{-1}(\heis)}\\
	&\leq \frac{|\gamma_2-\gamma_1|}{(1-\gamma_1)(1-\gamma_2)}\left(\|-\hlapl U_{\gamma_2,n}(t)\|_{\dot{H}^{-1}(\heis)}+\|D_sU_{\gamma_2,n}(t)\|_{\dot{H}^{-1}(\heis)}\right)\\
	&\leq \frac{|\gamma_2-\gamma_1|}{(1-\gamma_1)(1-\gamma_2)}2C(n).
\end{align*}
Finally, there exists $C'(n)>0$ such that
\[
\|\Pi^{(n)}(|U_{\gamma_1,n}|^2U_{\gamma_1,n}(t))-\Pi^{(n)}(|U_{\gamma_2,n}|^2U_{\gamma_2,n}(t))\|_{\dot{H}^{-1}(\heis)}
	\leq C'(n)\|R(t)\|_{\dot{H}^1(\heis)}.
\]

We define $f(t):=\|R(t)\|_{\dot{H}^1(\heis)}^2$ for $t\in\R$. Then there exists some constant $C''(n)>0$ such that
\begin{align*}
f'(t)
	&\leq 2\|\partial_tR(t)\|_{\dot{H}^1(\heis)}\|R(t)\|_{\dot{H}^1(\heis)}\\
	&\leq C''(n)\|\partial_tR(t)\|_{\dot{H}^{-1}(\heis)}\|R(t)\|_{\dot{H}^1(\heis)}\\
	&\leq C''(n)\left(\Big(\frac{2}{1-\gamma_1}+C'(n)\Big)\| R(t)\|_{\dot{H}^1(\heis)}^2+\frac{|\gamma_2-\gamma_1|}{(1-\gamma_1)(1-\gamma_2)}2C(n)\|R(t)\|_{\dot{H}^1(\heis)}\right)\\
	&\leq C''(n)\left(\Big(\frac{2}{1-\gamma_1}+C'(n)+\frac{|\gamma_2-\gamma_1|}{(1-\gamma_1)(1-\gamma_2)}C(n)\Big)\| R(t)\|_{\dot{H}^1(\heis)}^2+\frac{|\gamma_2-\gamma_1|}{(1-\gamma_1)(1-\gamma_2)}C(n)\right).
\end{align*}
Therefore, $f(t)$ satisfies a Gronwall-type inequality
\[
f(t)'\leq K_1(n)f(t)+K_2(n)\frac{|\gamma_2-\gamma_1|}{(1-\gamma_1)(1-\gamma_2)}
\]
with
\[
K_1(n)= C''(n)\left(\frac{2}{1-\gamma_1}+C'(n)+\frac{|\gamma_2-\gamma_1|}{(1-\gamma_1)(1-\gamma_2)}C(n)\right),
\]
and
\[
K_2(n)=C''(n)C(n).
\]
This inequality implies that for all $t\in\R$,
\[
f(t)\leq f(0)e^{K_1(n)|t|}+\frac{K_2(n)}{K_1(n)}\frac{|\gamma_2-\gamma_1|}{(1-\gamma_1)(1-\gamma_2)}(e^{K_1(n)|t|}-1)
\]
with
\[
f(0)=\|\Pi^{(n)}(U_0^{\gamma_1}-U_0^{\gamma_2})\|_{\dot{H}^1(\heis)}^2.
\]
Fix $t\in\R$ and $\gamma_1\in[\beta,1)$, we see that if $\gamma_2$ tends to $\gamma_1$, then $f(t)$ tends to $0$.
\end{proof}

Assume now that $\beta<\beta_0(n,t)=:\beta_0$. We find an upper bound for $\beta_0$ in $[\beta,1)$ independent on $n$ and $t$. The continuity of $\gamma\mapsto W_{\gamma,n}(t)$ implies that
\[
\|W_{\beta_0,n}(t)\|_{\dot{H}^1(\heis)}= r^2.
\]
The component of $U_{\beta,n}(t)$ along $V_0^+$ is bounded by
\begin{align*}
\|U_{\beta,n}^+(t)\|_{\dot{H}^1(\heis)}^2
	&\leq \mathcal{P}(U_{\beta,n}(t))\\
	&\leq \mathcal{P}(Q)+C_0r^2,
\end{align*}
therefore
\[
\|U_{\beta_0,n}(t)\|_{\dot{H}^1(\heis)}
	\leq r^2+(\mathcal{P}(Q)+C_0r^2)^\half,
\]
where $C:=r^2+(\mathcal{P}(Q)+C_0r^2)^\half$ does not depend on $n$ or $t$ anymore. Lemma \ref{lem:estimates_w_beta,n} now implies
\[
\|W_{\beta_0,n}(t)\|_{\dot{H}^1(\heis)}
	\leq C_1(1+C^2)\sqrt{1-\beta_0}.
\]
We conclude that
\[
r^2\leq C_1(1+C^2)\sqrt{1-\beta_0},
\]
which means
\[
\beta_0\leq 1-\left(\frac{r^2}{C_1(1+C^2)}\right)^2=:\beta^*(r),
\]
therefore $\beta<\beta^*(r)$. Taking the converse, we have proven that if $\beta\geq\beta^*(r)$, then $\beta_0=\beta$.

\begin{lem}\label{lem:estimates_u_beta,n}
There exists some constant $\beta^*(r)\in(0,1)$ such that if $\beta\geq\beta^*(r)$, then the solution $U_{\beta,n}$ to equation \eqref{eq:H_rescaled_n}
\begin{equation*}
\begin{cases}
i\partial_t U_{\beta,n} -\frac{\hlapl+\beta D_s}{1-\beta} U_{\beta,n}=\Pi^{(n)}(|U_{\beta,n}|^2U_{\beta,n})
\\
U_{\beta,n}(t=0)=U_0^{n}=\Pi^{(n)}(U_0)
\end{cases}
\end{equation*}
satisfies for all $t\in\R$
\[
\|W_{\beta,n}(t)\|_{\dot{H}^1(\heis)}
	\leq r^2
\]
and
\[
\|U_{\beta,n}^+(t)\|_{\dot{H}^1(\heis)}
	\leq (\mathcal{P}(Q)+C_0r^2)^\half.
\]
\end{lem}

We now show that $U_{\beta,n}(t)$ is close to the orbit $\M$ of $Q$ for $t\in\R$ and $\beta\geq\beta^*(r)$.

\begin{prop}
There exist $r_0>0$ and $c_0>0$ such that if $r<r_0$ and $\beta\geq\beta^*(r)$, then for all $t\in\R$,
\[
d(U_{\beta,n}(t),\M)< c_0r.
\]
\end{prop}

\begin{proof}
Fix $t\in\R$. It only remains to estimate $\delta(U_{\beta,n}^+(t))$ and apply Proposition \ref{prop:estimates_stability}.

On the one hand, since $(D_sW_{\beta,n}(t),W_{\beta,n}(t))_{L^2(\heis)}\leq \|W_{\beta,n}(t)\|_{\dot{H}^1(\heis)}^2\leq r^2$, the conservation of momentum leads to
\begin{equation}\label{ineq:delta_ubeta1}
|(D_sU_{\beta,n}^+(t),U_{\beta,n}^+(t))_{L^2(\heis)}-(D_sQ,Q)_{L^2(\heis)}|\leq (C_0+1)r^2.
\end{equation}

On the other hand, we estimate $|\|U_{\beta,n}^+(t)\|_{L^4(\heis)}^4-\|Q\|_{L^4(\heis)}^4|$ via the conservation of energy. We know that
\[
|\|U_{\beta,n}(t)\|_{L^4(\heis)}^4-\|U_{\beta,n}^+(t)\|_{L^4(\heis)}^4|
	\leq \|W_{\beta,n}(t)\|_{L^4(\heis)}(\|U_{\beta,n}(t)\|_{L^4(\heis)}+\|U_{\beta,n}^+(t)\|_{L^4(\heis)})^3.
\]
Since $\|U_{\beta,n}(t)\|_{\dot{H}^1(\heis)}$ is bounded thanks to Lemma \ref{lem:estimates_u_beta,n}, there exists $C_1>0$ such that
\begin{equation}\label{ineq:delta_ubeta2}
|\|U_{\beta,n}(t)\|_{L^4(\heis)}^4-\|U_{\beta,n}^+(t)\|_{L^4(\heis)}^4|
	\leq C_1\|W_{\beta,n}(t)\|_{\dot{H}^1(\heis)}
	\leq C_1r^2.
\end{equation}
Therefore, from \eqref{ineq:delta_ubeta1} and \eqref{ineq:delta_ubeta2}, we get
\begin{align*}
\mathcal{E}_\beta&(U_{\beta,n}(t))\\
	&= \half(-\frac{\hlapl+\beta D_s}{1-\beta}W_{\beta,n}(t),W_{\beta,n}(t))_{L^2(\heis)}+\half(D_sU_{\beta,n}^+(t),U_{\beta,n}^+(t))_{L^2(\heis)}-\frac{1}{4}\|U_{\beta,n}(t)\|_{L^4(\heis)}^4\\
	&\geq \half(D_sQ,Q)_{L^2(\heis)}-\frac{1}{4}\|U_{\beta,n}^+(t)\|_{L^4(\heis)}^4-(\frac{C_0+1}{2}+\frac{C_1}{4})r^2.
\end{align*}
However,
\[
\mathcal{E}_\beta(U_{\beta,n}(t))
	=\mathcal{E}_\beta(U_0^n)
	\leq \half(D_sQ,Q)_{L^2(\heis)}-\frac{1}{4}\|Q\|_{L^4(\heis)}^4+C_0r^2,
\]
therefore
\[
\frac{1}{4}\|U_{\beta,n}^+(t)\|_{L^4(\heis)}^4
	\geq \frac{1}{4}\|Q\|_{L^4(\heis)}^4- (\frac{3C_0+1}{2}+\frac{C_1}{4})r^2.
\]

For the reverse inequality, recall the link between $Q$ and the best constant in the embedding $\dot{H}^1(\heis)\cap V_0^+\hookrightarrow L^4(\heis)$ : if
\[
\inf_{u\in \dot{H}^1(\heis)\cap V_0^+}\frac{(D_su,u)_{L^2(\heis)}^2}{\|u\|_{L^4(\heis)}^4}=I_+,
\]
then
\[
(D_sQ,Q)_{L^2(\heis)}=\|Q\|_{L^4(\heis)}^4=I_+=\pi^2.
\]
This leads to
\begin{align*}
\|U_{\beta,n}^+(t)\|_{L^4(\heis)}^4
	&\leq \frac{1}{I_+}(D_sU_{\beta,n}^+(t),U_{\beta,n}^+(t))_{L^2(\heis)}^2\\
	&\leq \frac{1}{I_+}((D_sQ,Q)_{L^2(\heis)}+(C_0+1)r^2)^2\\
	&\leq  \frac{1}{I_+}(I_++(C_0+1)r^2)^2\\
	&\leq \|Q\|_{L^4(\heis)}^4+\frac{1}{I_+}(2I_+(C_0+1)+(C_0+1)^2r^2)r^2.
\end{align*}

In the end, we have proven that if $r\leq 1$, then for some constant $C_2>0$,
\[
\delta(U_{\beta,n}^+(t))\leq C_2r^2,
\]
and Proposition \ref{prop:estimates_stability} immediately implies that for $r$ small enough,
\[
d(U_{\beta,n}^+(t),\M)^2\leq CC_2r^2.
\]
Since $\|W_{\beta,n}(t)\|_{\dot{H}^1(\heis)}\leq r^2$, we get the Proposition.
\end{proof}

\subsection{Weak convergence}\label{subsection:limit_n_beta_fixed}

We now know that if $\beta\geq\beta_*(r)$, then for all $n\geq N$ and $t\in\R$,
\begin{equation}\label{ineq:distance_ubetan,Q}
d(U_{\beta,n}(t),\M)< c_0r.
\end{equation}
The aim is now to pass to the limit $n\to+\infty$ in equation \eqref{eq:H_rescaled_n}
\begin{equation*}
\begin{cases}
i\partial_t U_{\beta,n} -\frac{\hlapl+\beta D_s}{1-\beta} U_{\beta,n}=\Pi^{(n)}(|U_{\beta,n}|^2U_{\beta,n})
\\
U_{\beta,n}(t=0)=U_0^{n}=\Pi^{(n)}(U_0)
\end{cases}
\end{equation*}
and in inequality \eqref{ineq:distance_ubetan,Q} in order to get a weak solution $U_\beta$ to equation \eqref{eq:H_rescaled}
\begin{equation*}
\begin{cases}
i\partial_t U_{\beta} -\frac{\hlapl+\beta D_s}{1-\beta} U_{\beta}=|U_{\beta}|^2U_{\beta}
\\
U_\beta(t=0)=U_0
\end{cases}
\end{equation*}
which satisfies
\[
d(U_{\beta}(t),\M)\leq c_0r,
	\quad t\in\R.
\]

The method is identical to parts \ref{subsection:limiting_weak_convergence} and \ref{subsection:limiting_modulation} for the limiting equation : we use a uniform bound on $\|\partial_tU_{\beta,n}(t)\|_{\dot{H}^{-1}(\heis)}$. Thanks to Ascoli's theorem, the sequence $(U_{\beta,n})_{n\in\N}$ admits a weak limit $U_\beta$ which is a weak solution to \eqref{eq:H_rescaled}. Then, we construct bounded modulation parameters $X_{\beta,n}(t)$ in order to control the distance between $U_\beta$ and $\M$.

\begin{lem}
There exists $c_\beta>0$ such that for all $n\geq N$, $t\in\R$ and $X\in\R\times\T\times\R_+^*$,
\[
\|\partial_t(T_XU_{\beta,n})(t)\|_{\dot{H}^{-1}(\heis)}\leq c_\beta.
\]
\end{lem}
\begin{proof}
We know from Lemma \ref{lem:estimates_u_beta,n} that there exists some constant $C_1>0$ such that for all $n\geq N$ and $t\in\R$,
\[
\|U_{\beta,n}(t)\|_{\dot{H}^1(\heis)}\leq C_1.
\]

Set $V_{\beta,n}:=T_XU_{\beta,n}$. By symmetry invariance, $V_{\beta,n}$ satisfies that for all $n\geq N$ and $t\in\R$,
\[
\|V_{\beta,n}(t)\|_{\dot{H}^1(\heis)}\leq C_1.
\]
Moreover, $V_{\beta,n}$ is a solution to some equation
\begin{equation*}
\begin{cases}
i\partial_t V_{\beta,n} -\frac{\hlapl+\beta D_s}{1-\beta} V_{\beta,n}=\widetilde{\Pi}^{(n)}(|V_{\beta,n}|^2V_{\beta,n})
\\
V_{\beta,n}(t=0)=V_0^n=\widetilde{\Pi}^{(n)}(U_0)
\end{cases}.
\end{equation*}
The projector $\widetilde{\Pi}^{(n)}$ is defined as follows. Write $X=(s,\theta,\alpha)$. For $u\in\dot{H}^{-1}(\heis)$, we decompose
\[
u=\sum_{k\in\N}\sum_\pm\Pi_{k,\pm}(u)
\]
with $\Pi_{k,\pm}(u)\in \dot{H}^{-1}(\heis)\cap V_k^\pm$ for $(k,\pm)\in\N\times\{\pm\}$. Then
\[
\widehat{\widetilde{\Pi}^{(n)}(u)}(x,y,\sigma)=\sum_{k=0}^n\sum_\pm\widehat{\Pi_{k,\pm}(u)}(x,y,\sigma)\un_{\frac{\alpha^2}{n}\leq|\sigma|\leq\alpha^2n}.
\]

Thanks to the fact that $\widetilde{\Pi}^{(n)}$ is a projector and the embeddings $L^{\frac{4}{3}}(\heis)\hookrightarrow \dot{H}^{-1}(\heis)$ and $\dot{H}^1(\heis)\hookrightarrow L^4(\heis)$,
\begin{align*}
\|\partial_tV_{\beta,n}(t)\|_{\dot{H}^{-1}(\heis)}
	&\leq \frac{1}{1-\beta}\|-(\hlapl+\beta D_s)V_{\beta,n}\|_{\dot{H}^{-1}(\heis)}+\|\widetilde{\Pi}^{(n)}(|V_{\beta,n}|^2V_{\beta,n})\|_{\dot{H}^{-1}(\heis)}\\
	&\leq \frac{2}{1-\beta}\|-\hlapl V_{\beta,n}\|_{\dot{H}^{-1}(\heis)}+\||V_{\beta,n}|^2V_{\beta,n}\|_{\dot{H}^{-1}(\heis)}\\
	&\leq \frac{2}{1-\beta}\|V_{\beta,n}\|_{\dot{H}^1(\heis)}+K_1\||V_{\beta,n}|^2V_{\beta,n}\|_{L^{\frac{4}{3}}(\heis)}\\
	&\leq \frac{2}{1-\beta}\|V_{\beta,n}\|_{\dot{H}^1(\heis)}+K_2\|V_{\beta,n}\|_{\dot{H}^1(\heis)}^3\\
	&\leq \frac{2C_1}{1-\beta}+K_2C_1^3.
\end{align*}
\end{proof}

We deduce the weak convergence of $(U_{\beta,n})_{n\in\N}$, for which the proof is identical to part \ref{subsection:limiting_weak_convergence} and based on Ascoli's theorem.

\begin{lem}
Up to a subsequence, $(U_{\beta,n})_n$ converges weakly to a solution $U_\beta\in\classeC(\R,\dot{H}^1(\heis))$ (with the weak topology) to \eqref{eq:H_rescaled}
\begin{equation*}
\begin{cases}
i\partial_t U_{\beta} -\frac{\hlapl+\beta D_s}{1-\beta} U_{\beta}=|U_{\beta}|^2U_{\beta}
\\
U_\beta(t=0)=U_0
\end{cases}.
\end{equation*}
\end{lem}

Moreover, one can see that for all $X\in\R\times\T\times\R_+^*$ and $t_0,t\in\R$, setting $V_{\beta,n}:=T_{X^{-1}}U_{\beta,n}$,
\begin{align*}
\left|\frac{\d}{\d t}\|U_{\beta,n}(t)-T_XQ\|_{\dot{H}^1(\heis)}^2\right|
	&=\left|\frac{\d}{\d t}\|V_{\beta,n}(t)-Q\|_{\dot{H}^1(\heis)}^2\right|\\
	&=\left|2(\partial_t V_{\beta,n}(t),D_sQ)_{L^2(\heis)}\right|\\
	&\leq 2\|\partial_t V_{\beta,n}(t)\|_{\dot{H}^{-1}(\heis)}\|D_sQ\|_{\dot{H}^1(\heis)},
\end{align*}
which implies that for some constant $c_\beta>0$, for all $t_0,t\in\R$,
\[
\|U_{\beta,n}(t)-T_XQ\|_{\dot{H}^1(\heis)}^2
	\leq \|U_{\beta,n}(t_0)-T_XQ\|_{\dot{H}^1(\heis)}^2
+c_\beta|t-t_0|.
\]

Set $\varepsilon\in(0,1)$ and define $t_1:=\frac{(1+\varepsilon)^2-1}{c_\beta}(c_0r)^2$.
Note that $t_1$ may depend on $\beta$, but this is not important because in this part the varying parameter is $n$ whereas $\beta$ is fixed. The construction of $X_{\beta,n}$ as a piecewise constant functional is now the same as for the limiting system. For $k\in\Z$, $X_{\beta,n}$ is constant on $[kt_1,(k+1)t_1[$, equal to some $X_{\beta,n}^k\in\R\times\T\times\R_+^*$ to be chosen.  We first set $X_{\beta,n}^{-1}=X_{\beta,n}^0=(0,0,1)$. Then, at time $t_k=kt_1$, $k\geq 1$, we use the fact that $d(U_{\beta,n}(t_k),\M)< c_0r$ and choose $X_{\beta,n}^k$ such that $\|U_{\beta,n}(t_k)-T_{X_{\beta,n}^k}Q\|_{\dot{H}^1(\heis)}<c_0r$. By definition of $t_1$, for all $k\geq 0$ ad $t\in[t_k,t_k+t_1]$, $\|U_{\beta,n}(t)-T_{X_{\beta,n}^k}Q\|_{\dot{H}^1(\heis)}\leq (1+\varepsilon)c_0r$. We do a similar construction for negative times. The map $X_{\beta,n}$ satisfies
\begin{equation}\label{ineq:modulation_ubetan}
\|U_{\beta,n}(t)-T_{X_{\beta,n}(t)}Q\|_{\dot{H}^1(\heis)}\leq (1+\varepsilon)c_0r,
	\quad  t\in\R.
\end{equation}

It remains to show that $X_{\beta,n}$ is bounded independently of $n$ on bounded intervals. In order to do so, it is enough to control the gap between $X_{\beta,n}^{k-1}$ and $X_{\beta,n}^k$. By construction, at time $t_k$,
\[
\|U_{\beta,n}(t_k)-T_{X_{\beta,n}^{k-1}}Q\|_{\dot{H}^1(\heis)}\leq (1+\varepsilon)c_0r
\]
and
\[
\|U_{\beta,n}(t_k)-T_{X_{\beta,n}^k}Q\|_{\dot{H}^1(\heis)} < c_0r,
\]
therefore
\[
\|T_{X_n^{k-1}}Q-T_{X_n^k}Q\|_{\dot{H}^1(\heis)}\leq (2+\varepsilon)c_0r.
\]
Using Lemma \ref{lem:gap_TXQ_Q}, we conclude that if $r\leq r_0$ is small enough (for example if $3c_0r_0\leq \sqrt{\pi}r_1$), then
\[
|X_n^{k-1}(X_n^k)^{-1}|\leq c_1.
\]

Now, for fixed $t\in\R$, the sequence $(X_{\beta,n}(t))_{n\in\N}$ is bounded, therefore up to extraction, this sequence converges to some $X_\beta(t)\in\R\times\T\times\R_+^*$, and passing to the weak limit in \eqref{ineq:modulation_ubetan},
\[
\|U_\beta(t)-T_{X_\beta(t)}Q\|_{\dot{H}^1(\heis)}\leq(1+\varepsilon)c_0r.
\]

\section{Appendix : Rate of convergence of \texorpdfstring{$Q_\beta$}{Q_beta} to \texorpdfstring{$Q$}{Q}}\label{appendix}

In order to conclude the proof of Theorem \ref{thm:stability_Qbeta_precise}, it only remains to make precise
\[
\delta_\beta(Q)=|\mathcal{E}_\beta(Q_\beta)-\mathcal{E}(Q)|+|\mathcal{P}(Q_\beta)-\mathcal{P}(Q)|,
\]
which vanishes as $\beta$ tends to $1$.

Decompose
\[
Q_\beta=Q_\beta^++R_\beta,
\]
where $Q_\beta^+\in \dot{H}^1(\heis)\cap V_0^+$ and $R_\beta\in\bigoplus_{(k,\pm)\neq(0,+)}\dot{H}^1(\heis)\cap V_k^\pm$.
We improve the bound from \cite{schroheis}
\[
\delta(Q_\beta^+)+\|R_\beta\|_{\dot{H}^1(\heis)}=\gdO((1-\beta)^\half).
\]

\begin{prop}
Let $\varepsilon>0$. Then as $\beta$ tends to $1$,
\[
\delta(Q_\beta^+)+\|R_\beta\|_{\dot{H}^1(\heis)}=\gdO((1-\beta)^{2-\varepsilon}),
\]
which implies that
\[
\|Q_\beta-Q\|_{\dot{H}^1(\heis)}=\gdO((1-\beta)^{1-\frac{\varepsilon}{2}}).
\]
\end{prop}

\begin{proof}
Assume that we have proven that
\[
\delta(Q_\beta^+)+\|R_\beta\|_{\dot{H}^1(\heis)}=\gdO((1-\beta)^\gamma)
\]
for some exponent $\gamma>0$ (for example we already know that it is true for $\gamma=\frac{1}{2}$), and therefore
\[
\|Q_\beta-Q\|_{\dot{H}^1(\heis)}=\gdO((1-\beta)^{\frac{\gamma}{2}}).
\]
We increase the exponent $\gamma$ by showing that actually
\[
\delta(Q_\beta^+)+\|R_\beta\|_{\dot{H}^1(\heis)}=\gdO((1-\beta)^{1+\frac{\gamma}{2}}).
\]

Indeed, since $R_\beta\in\bigoplus_{(k,\pm)\neq(0,+)}\dot{H}^1(\heis)\cap V_k^\pm$, the norms $\|R_\beta\|_{\dot{H}^1(\heis)}$ and $\|-(\hlapl+\beta D_s)R_\beta\|_{\dot{H}^{-1}(\heis)}$ are equivalent
\begin{align*}
\|R_\beta\|_{\dot{H}^1(\heis)}
	&\leq 2\|-(\hlapl+\beta D_s)R_\beta\|_{\dot{H}^{-1}(\heis)}.
\end{align*}
Projecting the equation satisfied by $Q_\beta$
\[
-\frac{\hlapl+\beta D_s}{1-\beta}Q_\beta=|Q_\beta|^2Q_\beta
\]
on $\bigoplus_{(k,\pm)\neq(0,+)}\dot{H}^1(\heis)\cap V_k^\pm$, we deduce that
\begin{align*}
\|R_\beta\|_{\dot{H}^1(\heis)}
	&\leq 2(1-\beta)\|(\id-\Pi_{0,+})(|Q_\beta|^2Q_\beta)\|_{\dot{H}^{-1}(\heis)}.
\end{align*}
Since $|Q|^2Q=D_sQ\in \dot{H}^{-1}(\heis)\cap V_0^+$, one can make this term appear inside the right term of the inequality :
\begin{align*}
\|R_\beta\|_{\dot{H}^1(\heis)}
	&\leq 2(1-\beta)\|(\id-\Pi_{0,+})(|Q_\beta|^2Q_\beta-|Q|^2Q)\|_{\dot{H}^{-1}(\heis)}\\
	&\leq 2K(1-\beta)\|(\id-\Pi_{0,+})(|Q_\beta|^2Q_\beta-|Q|^2Q)\|_{L^{\frac{4}{3}}(\heis)}.
\end{align*}
Now, since $(\id-\Pi_{0,+})$ defines a bounded operator on $L^{\frac{4}{3}}(\heis)$, there exist $C_1,C_2>0$ such that
\begin{align*}
\|R_\beta\|_{\dot{H}^1(\heis)}
	&\leq C_1(1-\beta)\||Q_\beta|^2Q_\beta-|Q|^2Q\|_{L^{\frac{4}{3}}(\heis)}\\
	&\leq C_2(1-\beta)\|Q_\beta-Q\|_{\dot{H}^1(\heis)}(\|Q_\beta\|_{\dot{H}^1(\heis)}+\|Q\|_{\dot{H}^1(\heis)})^2.
\end{align*}
But since $(Q_\beta)_\beta$ is bounded in $\dot{H}^1(\heis)$, we get that for some $C_3>0$,
\begin{align*}
\|R_\beta\|_{\dot{H}^1(\heis)}
	&\leq C_3(1-\beta)\|Q_\beta-Q\|_{\dot{H}^1(\heis)}\\
	&=\gdO((1-\beta)^{1+\frac{\gamma}{2}}).
\end{align*}

Therefore,
\begin{align*}
|\|Q_\beta^+\|_{\dot{H}^1(\heis)}^2-\|Q_\beta\|_{\dot{H}^1(\heis)}^2|
	&\leq 2\|R_\beta\|_{\dot{H}^1(\heis)}(\|R_\beta\|_{\dot{H}^1(\heis)}+\|Q_\beta\|_{\dot{H}^1(\heis)})\\
	&= \gdO((1-\beta)^{1+\frac{\gamma}{2}})
\end{align*}
and
\begin{align*}
|\|Q_\beta^+\|_{L^4(\heis)}^4-\|Q_\beta\|_{L^4(\heis)}^4|
	&\lesssim \|R_\beta\|_{L^4(\heis)}(\|R_\beta\|_{L^4(\heis)}+\|Q_\beta\|_{L^4(\heis)})^3\\
	&= \gdO((1-\beta)^{1+\frac{\gamma}{2}}),
\end{align*}
which means that
\[
\delta(Q_\beta^+)= \gdO((1-\beta)^{1+\frac{\gamma}{2}}).
\]
It now remains to consider the sequence $\gamma_{n+1}=1+\frac{\gamma_n}{2}$, $\gamma_0=\half$, which is convergent to $2$.

\end{proof}

\bibliography{mybib}{}

\begin{thebibliography}{10}

\bibitem{Aubin1976}
T.~Aubin.
\newblock {Problèmes isopérimétriques et espaces de Sobolev}.
\newblock {\em J. Differential Geom.}, 11(4):573--598, 1976.

\bibitem{BahouriFermanianGallagher2016}
H.~Bahouri, C.~Fermanian-Kammerer, and I.~Gallagher.
\newblock {Dispersive estimates for the Schr{\"o}dinger operator on step-2
  stratified Lie groups}.
\newblock {\em Analysis and PDE}, 9(3):545--574, 2016.

\bibitem{BahouriGerardXu2000}
H.~Bahouri, P.~G{\'e}rard, and C.-J. Xu.
\newblock {Espaces de Besov et estimations de Strichartz
  g{\'e}n{\'e}ralis{\'e}es sur le groupe de Heisenberg}.
\newblock {\em Journal d'Analyse Math{\'e}matique}, 82(1):93--118, Dec 2000.

\bibitem{BellazziniGeorgievLenzmannVisciglia2019}
J.~Bellazzini, V.~Georgiev, E.~Lenzmann, and N.~Visciglia.
\newblock On traveling solitary waves and absence of small data scattering for
  nonlinear half-wave equations.
\newblock {\em Communications in Mathematical Physics}, Feb 2019.

\bibitem{BellazziniGeorgievVisciglia2018}
J.~Bellazzini, V.~Georgiev, and N.~Visciglia.
\newblock Long time dynamics for semi-relativistic nls and half wave in
  arbitrary dimension.
\newblock {\em Mathematische Annalen}, 371(1):707--740, Jun 2018.

\bibitem{Burq2005}
N.~Burq, P.~G{\'e}rard, and N.~Tzvetkov.
\newblock {Bilinear eigenfunction estimates and the nonlinear Schrödinger
  equation on surfaces}.
\newblock {\em Inventiones mathematicae}, 159(1):187--223, 2005.

\bibitem{bergman_projectors}
D.~Békollé, A.~Bonami, G.~Garrigós, C.~Nana, M.~Peloso, and F.~Ricci.
\newblock {Lecture Notes on Bergman projectors in tube domains over cones : an
  analytic and geometric viewpoint}.
\newblock {\em IMHOTEP: African Journal of Pure and Applied Mathematics}, 5(0),
  2012.

\bibitem{DeLellisSzekelyhidi2009}
C.~De~Lellis and L.~Sz{\'e}kelyhidi~Jr.
\newblock {The Euler equations as a differential inclusion}.
\newblock {\em Annals of Mathematics}, 170:1417--1436, 2009.

\bibitem{DelHierro2005}
M.~Del~Hierro.
\newblock {Dispersive and Strichartz estimates on H-type groups}.
\newblock {\em Studia Mathematica}, 1(169):1--20, 2005.

\bibitem{KenigMerle2006}
C.~E.~Kenig and F.~Merle.
\newblock {Global well-posedness, scattering and blow-up for the
  energy-critical, focusing, non-linear Schrödinger equation in the radial
  case}.
\newblock {\em Inventiones mathematicae}, 166:645--675, 01 2006.

\bibitem{schroheis}
L.~Gassot.
\newblock {On the radially symmetric traveling waves for the Schrödinger
  equation on the Heisenberg group}.
\newblock {\em arXiv:1904.07010}, 2019.

\bibitem{GerardGrellier2008}
P.~G{\'e}rard and S.~Grellier.
\newblock {L'équation de Szegő cubique}.
\newblock {\em {Séminaire X-EDP, École Polytechnique}}, 2008.

\bibitem{GerardLenzmannPocovnicuRaphael2018}
P.~G{\'e}rard, E.~Lenzmann, O.~Pocovnicu, and P.~Rapha{\"e}l.
\newblock A two-soliton with transient turbulent regime for the cubic half-wave
  equation on the real line.
\newblock {\em Annals of PDE}, 4(1):7, 2018.

\bibitem{KriegerLenzmannRaphael2013}
J.~Krieger, E.~Lenzmann, and P.~Raphaël.
\newblock {Nondispersive solutions to the $L^2$-critical half-wave equation}.
\newblock {\em Archive for Rational Mechanics and Analysis}, 209(1):61--129,
  2013.

\bibitem{stein_harmonic}
E.~M.~Stein and T.~S.~Murphy.
\newblock {\em {Harmonic analysis : real-variable methods, orthogonality, and
  oscillatory integrals.}}
\newblock Princeton University Press, 1993.

\bibitem{Pocovnicu2012}
O.~Pocovnicu.
\newblock {Soliton interaction with small Toeplitz potentials for the Szegő
  equation on $\mathbb R$}.
\newblock {\em Dynamics of Partial Differential Equations}, 9(1):1--27, 2012.

\bibitem{Talenti1976}
G.~Talenti.
\newblock {Best constant in Sobolev inequality}.
\newblock {\em Annali di Matematica Pura ed Applicata}, 110(1):353--372, Dec
  1976.

\end{thebibliography}
\bibliographystyle{abbrv}

\Addresses

\end{document}